\documentclass[12]{amsart}
\usepackage{amscd,amssymb}
\usepackage[all]{xy}
\usepackage{graphicx}
\usepackage[colorlinks,plainpages,backref,urlcolor=blue]{hyperref}

\newtheorem{theorem}{Theorem}[section]
\newtheorem{corollary}[theorem]{Corollary}
\newtheorem{lemma}[theorem]{Lemma}
\newtheorem{prop}[theorem]{Proposition}
\newtheorem{conjecture}[theorem]{Conjecture}

\theoremstyle{definition}
\newtheorem{definition}[theorem]{Definition}
\newtheorem{example}[theorem]{Example}
\newtheorem{remark}[theorem]{Remark}
\newtheorem{question}[theorem]{Question}
\newtheorem*{ack}{Acknowledgments}

\newcommand{\N}{\mathbb{N}}
\newcommand{\Z}{\mathbb{Z}}
\newcommand{\Q}{\mathbb{Q}}

\newcommand{\C}{\mathbb{C}}
\renewcommand{\L}{\mathbb{L}}
\newcommand{\TT}{\mathbb{T}}
\newcommand{\PP}{\mathbb{P}}

\renewcommand{\k}{\Bbbk}
\newcommand{\RR}{{\mathcal R}}
\newcommand{\VV}{{\mathcal V}}
\newcommand{\A}{{\mathcal{A}}}

\newcommand{\B}{{\mathfrak{B}}}
\newcommand{\g}{{\mathfrak{g}}}
\newcommand{\h}{{\mathfrak{h}}}
\newcommand{\gl}{{\mathfrak{gl}}}
\renewcommand{\sl}{{\mathfrak{sl}}}
\newcommand{\m}{{\mathfrak{m}}}
\renewcommand{\ss}{{\mathfrak{s}}}
\newcommand{\dd}{{\mathfrak{d}}}

\newcommand{\G}{{\Gamma}}
\newcommand{\T}{{\mathcal{T}}}
\newcommand{\wT}{\widetilde{\T}}
\newcommand{\wF}{{\widetilde{F}}}
\newcommand{\wJ}{{\widetilde{J}}}
\newcommand{\PS}{{{\rm P}\Sigma}}

\DeclareMathOperator{\rank}{rank}
\DeclareMathOperator{\gr}{gr}
\DeclareMathOperator{\im}{im}
\DeclareMathOperator{\coker}{coker}

\DeclareMathOperator{\id}{id}
\DeclareMathOperator{\ab}{{ab}}
\DeclareMathOperator{\Sym}{Sym}
\DeclareMathOperator{\ch}{char}
\DeclareMathOperator{\GL}{GL}
\DeclareMathOperator{\SL}{SL}
\DeclareMathOperator{\Sp}{Sp}
\DeclareMathOperator{\Hom}{{Hom}}

\DeclareMathOperator{\ann}{{ann}}

\DeclareMathOperator{\ev}{ev}

\DeclareMathOperator{\Out}{Out}
\DeclareMathOperator{\Inn}{Inn}
\DeclareMathOperator{\Aut}{Aut}
\DeclareMathOperator{\OA}{OA}
\DeclareMathOperator{\IA}{IA}

\DeclareMathOperator{\weight}{weight}
\DeclareMathOperator{\Der}{{Der}}
\DeclareMathOperator{\ad}{ad}
\DeclareMathOperator{\Ad}{Ad}
\DeclareMathOperator{\Lie}{Lie}
\DeclareMathOperator{\TC}{TC}
\DeclareMathOperator{\supp}{supp}
\DeclareMathOperator{\specm}{Specm}

\DeclareMathOperator{\height}{height}

\newcommand{\same}{\Longleftrightarrow}
\newcommand{\surj}{\twoheadrightarrow}
\newcommand{\inj}{\hookrightarrow}

\newcommand{\isom}{\xrightarrow{\,\simeq\,}}
\newcommand{\abs}[1]{\left| #1 \right|}

\newcommand{\cv}{\check{{\VV}}}
\makeatletter
  \AtBeginDocument{%
    \renewcommand{\@listi}
    {\setlength{\leftmargin}{\leftmargini}
    \setlength{\topsep} {2pt}
     \setlength{\parsep} {\parskip}
     \setlength{\itemsep}{2.0pt}}}
\makeatother	     

\numberwithin{equation}{section} 

\begin{document}


\title[Homological finiteness in the Johnson filtration]{%
Homological finiteness in the Johnson filtration of \\ 
the automorphism group of a free group}

\author[Stefan Papadima]{Stefan Papadima$^1$}
\address{Simion Stoilow Institute of Mathematics, 
P.O. Box 1-764,
RO-014700 Bucharest, Romania}
\email{Stefan.Papadima@imar.ro}
\thanks{$^1$Partially supported by 
PN-II-ID-PCE-2011-3-0288, grant 132/05.10.2011}

\author[Alexander~I.~Suciu]{Alexander~I.~Suciu$^2$}
\address{Department of Mathematics,
Northeastern University,
Boston, MA 02115, USA}
\email{a.suciu@neu.edu}
\thanks{$^2$Partially supported by NSA grant H98230-09-1-0021 
and NSF grant DMS--1010298}

\subjclass[2010]{Primary
20E36, 20J05. 
Secondary
20F14, 20G05, 55N25.
}

\keywords{Automorphism group of free group, Torelli group, 
Johnson filtration, Johnson homomorphism, resonance variety, 
characteristic variety, Alexander invariant.}

\begin{abstract}
We examine the Johnson filtration of the (outer) automorphism group 
of a finitely generated group. In the case of a free group, we find a 
surprising result: the first Betti number of the second subgroup in 
the Johnson filtration is finite. Moreover, the corresponding 
Alexander invariant is a module with non-trivial action over the 
Laurent polynomial ring. In the process, we show that the first 
resonance variety of the outer Torelli group of a free group 
is trivial.  We also establish a general relationship between 
the Alexander invariant and its infinitesimal counterpart.  
\end{abstract}

\maketitle
\tableofcontents

\newpage
\section{Introduction}
\label{sect:intro}

\subsection{Overview}
\label{subsec:intro setup}

Let $F_n$ be the free group of rank $n$, and let $J^s_n$ 
be the $s$-th term of the Andreadakis--Johnson filtration 
on the automorphism group $\Aut(F_n)$. In a recent paper, 
F.~Cohen, A.~Heap, and A.~Pettet formulated the following 
conjecture. 

\begin{conjecture}[\cite{CHP}]
\label{conj:chp}
If $n\ge 3$, $s\ge 2$, and $1\le i\le n-2$, the cohomology group 
$H^i(J^s_n,\Z)$ is not finitely generated.
\end{conjecture}

In this note, we disprove this conjecture, at least rationally, in 
the case when $n\ge 5$, $s=2$, and $i=1$.  One of our main 
results can be stated as follows. 

\begin{theorem}
\label{thm:intro1}
If $n\ge 5$, then $\dim_{\Q} H^1 (J^2_n, \Q)<\infty$. 
\end{theorem}

The theorem follows at once from Corollary \ref{cor:jprime}\eqref{jo1} 
and Theorem \ref{thm:fhp} towards the end of the paper. To arrive 
at the result, though, we need to build a fair amount of machinery, 
and introduce a number of inter-related concepts, of a rather general 
nature.  The purpose of this introductory section is to provide a 
guide to the technology involved, and to sketch the architecture 
of the proof.  Along the way, we state some of the other results 
we obtain here, and indicate the context in which those results fit.

\subsection{Torelli group and Johnson filtration}
\label{subsec:intro torelli}

Given any group $G$, the commutator defines a descending 
filtration, $\{ \G^s (G)\}_{s\ge 1}$, known as the lower central series 
of $G$.  Taking the direct sum of the successive quotients in this series 
yields the associated graded Lie algebra, $\gr_{\G} (G)$. 

On the automorphism group $\Aut (G)$, there is 
another descending filtration, $\{ F^s\}_{s\ge 0}$, 
called the {\em Johnson filtration}: an automorphism belongs to $F^s$ if 
and only if it has the same `$s$-jet' as the identity, with respect to the
lower central series of $G$. Clearly, $F^0= \Aut (G)$. The next term, 
$F^1$---called the {\em Torelli group}, and denoted $\T_G$---plays 
a crucial role in our investigation. 

We emphasize throughout the paper various equivariance 
properties with respect to  the {\em symmetry group},  
$\A (G)=F^0/F^1$.   For instance, there are inclusions, 
$\G^s(\T_G) \subseteq F^s$, for all $s\ge 1$, giving rise 
to a graded Lie algebra morphism, 
$\iota_F\colon \gr_{\G} (\T_G) \to \gr_{F} (\T_G)$,   
which is readily seen to be equivariant with respect 
to naturally defined actions of $\A(G)$ on source and target. 

The graded Lie algebra of derivations, $\Der (\gr_{\G} (G))$, may be
viewed as an infinitesimal approximation to the Torelli group. 
This philosophy goes back to a remarkable sequence of papers 
starting with \cite{J80}, where 
D.~Johnson introduced and studied what is now known as the 
{\em Johnson homomorphism}. We show that his construction 
works for an arbitrary group $G$, giving an
$\A (G)$-equivariant embedding of graded Lie algebras, 
$J \colon \gr_{F} (\T_G) \inj \Der (\gr_{\Gamma} (G))$.

We are  guided in our study by known results about the 
Torelli group $\T_g=\T_{\pi_1(\Sigma_g)}$ associated to   
the fundamental group of a closed Riemann surface of genus $g$.  
In this case, a key role is played by the symmetry group 
$\A(\pi_1(\Sigma_g))$, which may be identified with the symplectic 
group $\Sp(2g,\Z)$.  A delicate argument due to S.~Morita \cite{Mo}  
shows that $\Gamma^s (\T_g)$ can have infinite index in $F^s$.

When $G$ is the free group of rank $n$, the Torelli group $\T_{F_n}$ 
is known as the group of IA-automorphisms, denoted $\IA_n$,   
while the symmetry group $\A (F_n)$ is simply $\GL_n (\Z)$.  
Computations by Andreadakis \cite{An}, Cohen--Pakianathan, Farb, 
and Kawazumi, summarized by Pettet in \cite{Pe} show that 
$J_n^2$, the second term of the Johnson filtration of $\Aut (F_n)$, 
is equal to $\IA_n'$, the derived subgroup of $\IA_n$, 
for all $n\ge 3$.

\subsection{Outer automorphisms}
\label{subsec:intro outer}

Similar considerations apply to the outer automorphism group, 
$\Out (G)=\Aut (G)/\Inn (G)$.  The quotient Johnson filtration 
and the outer Torelli group are denoted by $\{ \wF^s \}$ and $\wT_G$, 
respectively. 

When the associated graded Lie algebra of $G$ is 
centerless, we construct an ``outer Johnson homomorphism", 
$\widetilde{J} \colon \gr_{\wF} (\wT_G) \inj 
\widetilde{\Der} (\gr_{\G} (G))$, 
which fits into the following commuting diagram 
in the category of graded Lie algebras endowed with a 
compatible $\A(G)$-module structure,
\begin{equation}
\label{eq:intro cd}
\xymatrixcolsep{36pt}
\xymatrix{
\gr_{\G} (G)  \ar^{=}[r]  \ar@{^{(}->}[d]  
& \gr_{\Gamma} (G)  \ar^{=}[r]  \ar@{^{(}->}[d]  
& \gr_{\Gamma} (G) \ar@{^{(}->}[d]\\  
\gr_{\G} (\T_G) \ar^(.4){\iota_F}[r]  \ar@{->>}[d] 
&\gr_{F} (\T_G) \ar@{^{(}->}^(.4){J}[r]  \ar@{->>}[d]  
& \Der (\gr_{\Gamma} (G)) \ar@{->>}[d]\\
\gr_{\G} (\wT_G) \ar^(.4){\iota_{\wF}}[r]  
&\gr_{\wF} (\wT_G) \ar@{^{(}->}^(.4){\widetilde{J}}[r]  
& \widetilde{\Der} (\gr_{\Gamma} (G))\, ,
}
\end{equation}
where the top vertical arrows are given by the relevant adjoint 
maps, the bottom vertical arrows are canonical projections, 
and the columns are all exact.

It turns out that $J\circ \iota_F$ and
$\widetilde{J}\circ \iota_{\wF}$ are simultaneously injective 
or surjective, in any fixed degree. Furthermore, the injectivity of
$J\circ \iota_F$, up to a given degree $q$,  characterizes the 
equalities $\G^s(\T_G) = F^s$ and  $\G^s(\wT_G) = \wF^s$, for $s\le q+1$.  
In Corollary \ref{cor:exot}, we obtain the following useful consequence. 

\begin{prop}
\label{prop:46intro}
Let $G$ be a residually nilpotent group. Suppose that 
$Z(\gr_{\G} (G))=0$ and $J\circ \iota_F$ is injective, up to degree $q$.
Then, for each $s\le q+1$ there is an exact sequence
\begin{equation}
\label{eq:grintro}
\xymatrix{1\ar[r]& \G^s (G)\ar[r]
& \G^s (\T_G)\ar[r]& \G^s (\wT_G) \ar[r]& 1} .
\end{equation}
\end{prop}

Our main interest is in the outer Torelli group 
associated to the free group of rank $n\ge 3$.    
This group, $\OA_n:=\wT_{F_n}$, fits into the exact sequence 
\begin{equation}
\label{eq:deftintro}
\xymatrix{1\ar[r]& \OA_n \ar[r]& \Out (F_n) \ar[r]& \GL_n (\Z) \ar[r]& 1}.
\end{equation}
As is well-known, the free group $F_n$ is residually nilpotent;  
moreover, its associated graded Lie algebra is a free Lie algebra, 
and thus centerless. In view of Proposition \ref{prop:46intro}, 
we obtain the exact sequence 
\begin{equation}
\label{eq:gexintro}
\xymatrix{1 \ar[r]& F_n' \ar[r]& \IA_n' \ar[r]& \OA_n'  \ar[r]& 1} .
\end{equation}

\subsection{Alexander invariant and characteristic varieties}
\label{subsec:intro alexcv}

Returning to the general situation, let $G$ be a group, 
and consider its {\em Alexander invariant}, $B(G)=H_1 (G', \Z)$.  
This is a classical object, with roots in the universal 
abelian covers arising in low-dimensional topology. 
The group $B(G)$ carries the structure of a module 
over the group ring of the abelianization, $R=\Z{G}_{\ab}$, 
and comes endowed with a $\Z$-linear action of $\Aut (G)$.   
As shown in Proposition \ref{prop:rlin}, 
restricting this action to the Torelli group $\T_G$ preserves 
the $R$-module structure on $B(G)$. 

Assume now that $G$ is finitely generated, so that $B(G)$ 
is also finitely generated as an $R$-module. As noted in 
Proposition \ref{prop:alex gr}, if the $\Q{G}_{\ab}$-module 
$B(G)\otimes \Q$ has trivial $G_{\ab}$-action, then 
$\gr_{\G}^3 (G)\otimes \Q=0$, and thus 
$\dim_{\Q} \gr_{\G} (G)\otimes \Q <\infty$.

Let $\TT (G)= \Hom (G_{\ab}, \C^{\times})$ be the character group, 
parametrizing complex one-dimensional representations of $G$. 
The {\em characteristic varieties}\/ $\VV^i_d (G)\subseteq \TT (G)$ 
are the jump loci for the homology of $G$ with such
coefficients; see \S\ref{subsec:scv} for full details. 
We are mostly interested here in the set 
$\VV(G)=\VV^1_1(G)$; since $G$ is assumed to 
be finitely generated, this is a Zariski closed subset  
of $\TT (G)$.   In fact, away from the identity $1\in \TT (G)$, 
the characteristic variety $\VV(G)$ coincides with the 
support variety of the module $B(G)\otimes \C$.

More generally, one can talk about the characteristic varieties 
of a connected CW-complex $X$.  It turns out that these varieties 
control the homological finiteness properties of abelian covers of 
$X$. The precise statement is to be found in Theorem \ref{thm:dfext}, 
an extension of a basic result due to Dwyer and Fried \cite{DF}. 
For our purposes here, the following consequence will be 
particularly important. 

\begin{prop}
\label{prop:vtestintro}
Let $G$ be a finitely generated group. Then 
$\dim_{\Q} B(G)\otimes \Q <\infty$ if and only if $\VV (G)$ is finite.
\end{prop}

Returning now to the free group $F_n$ and the Torelli groups 
$\IA_n$ and $\OA_n$, it is known from the work of Magnus 
that both these groups are finitely generated.  Moreover, 
as mentioned previously, $\IA_n'=J_n^2$.  Thus, 
Theorem \ref{thm:intro1} can be rephrased as saying that 
$\dim_{\Q} B(\IA_n)\otimes \Q <\infty$, for all $n\ge 5$. 

For technical reasons, we find it easier to work first with 
the group $\OA_n$, and prove the analogous statement 
for this group.  In view of Proposition \ref{prop:vtestintro}, 
it is enough to show that $\VV (\OA_n)$ is a finite set. 
To establish this fact, we follow the strategy developed 
in \cite{DP}, based on symmetry.  The first step is to note that 
extension \eqref{eq:deftintro} yields natural actions of the discrete 
group $\SL_n (\Z)$ on both the lattice $L= (\OA_n)_{\ab}$ and 
on the algebraic torus $\TT (L)$, so that the latter action leaves 
the subvariety $\VV (\OA_n) \subseteq \TT (L)$ invariant.

\subsection{Infinitesimal Alexander invariant and resonance varieties}
\label{subsec:intro res}

To progress further, we turn to some very useful approximations 
to the Alexander invariant and characteristic varieties of a finitely 
generated group $G$.

The {\em infinitesimal Alexander invariant}, $\B (G)$, 
is a complex vector space that approximates $B(G)\otimes \C$.  
By definition, $\B(G)$ is a finitely presented module over the 
polynomial ring $S=\Sym(G_{\ab}\otimes \C)$, with presentation 
derived from the co-restriction to the image of the cup-product 
map, $\cup_G\colon H^1(G,\C)\wedge H^1(G,\C)\to H^2(G,\C)$.

The {\em resonance varieties}, $\RR^i_d (G)$, 
approximate the corresponding characteristic varieties of $G$;   
they sit inside the tangent space at the origin to $\TT (G)$, 
which may be identified with the vector space 
$\Hom (G_{\ab}, \C)= H^1(G, \C)$.  Our main 
interest here is in the set $\RR(G)=\RR^1_1(G)$, 
defined in \eqref{eq:rv},
which is a Zariski closed affine cone.  The resonance variety 
$\RR(G)$ depends only on the co-restriction to the image of 
$\cup_G$, and coincides with the support of the module $\B(G)$, 
at least away from the origin $0\in H^1(G,\C)$.

Typically, the computation of the resonance varieties is 
much more manageable than that of the characteristic varieties. 
Nevertheless, knowledge of the former often sheds light on the 
latter. For instance, if $\RR (G)\subseteq \{ 0\}$, then 
either $1\notin \VV(G)$, or $1$ is an isolated point in $\VV(G)$; 
see Corollary \ref{cor:zerores}. 

If, in fact, $\VV (G) \subseteq \{1\}$, then more can be said.  
To start with, Proposition \ref{prop:vtestintro} guarantees 
that the complexified Alexander invariant, $B(G)\otimes \C$, 
has finite dimension as a $\C$-vector space. Furthermore, 
the result below (proved in Theorem \ref{thm:cv1}) relates 
the dimension of $B(G)\otimes \C$  to the dimension of its 
infinitesimal analogue, $\B(G)$.

\begin{theorem}
\label{thm:cvintro}
Let $G$ be a finitely generated group. Suppose  
$\VV (G) \subseteq \{1\}$.  Then, the following hold.
\begin{enumerate}
\item \label{v1intro}
$\dim_{\C} B(G)\otimes \C \le \dim_{\C} \B(G)$.

\item \label{v2intro}
If $\RR (G) \subseteq \{0\}$, then 
$\dim_{\C} B(G)\otimes \C \le \dim_{\C} \B(G) <\infty$.

\item \label{v3intro}
If $G$ is $1$-formal, then 
$\dim_{\C} B(G)\otimes \C = \dim_{\C} \B(G) <\infty$.
\end{enumerate}
\end{theorem}

Here, the $1$-formality of a group is considered in the sense of 
D. Sullivan \cite{Su}.   For instance, by the main result of 
R.~Hain from \cite{Ha}, the Torelli group $\wT_g= \wT_{\pi_1(\Sigma_g)}$ 
is $1$-formal, for $g\ge 6$.  The $1$-formality assumption 
in part \eqref{v3intro} is needed in order to be able to 
invoke the Tangent Cone theorem from \cite{DPS-duke}, 
which states that $\TC_1(\VV(G))=\RR(G)$ in this case.

Returning now to the Torelli groups of the free group $F_n$, 
we obtain in Theorem \ref{thm:gr ia} and Corollary \ref{cor:gr ia} the 
following result. 

\begin{theorem}
\label{thm:intro2}
Let $G$ be either $\IA_n$ or $\OA_n$.  If 
$n\ge 3$, then $\gr_{\G}(G) \otimes \Q$ is infinite 
dimensional (as a $\Q$-vector space), and 
$B(G)\otimes \Q$ has non-trivial $G_{\ab}$-action.
\end{theorem}

This result is similar to Proposition 9.5 from \cite{Ha}, where it is shown 
that $\dim_{\Q} \gr_{\G}(\wT_{g})\otimes \Q =\infty$, for $g\ge 3$.

\subsection{Homological finiteness for $\OA_n'$ and $\IA_n'$}
\label{ssi5}

We are now in a position to explain how the proof of 
Theorem \ref{thm:intro1} is completed.  We start with 
the analogous result for $\OA_n$.  As noted above, 
it is enough to show that $\VV(\OA_n)$ is finite. The 
first step is to establish that $\RR (\OA_n)=\{ 0\}$.

As explained by Pettet in \cite{Pe}, the group $L=(\OA_n)_{\ab}$ is free 
abelian, and the $\SL_n (\Z)$-representation in $T_1 \TT (L)= (L\otimes \C)^*$
coming from \eqref{eq:deftintro} extends to a rational, {\em irreducible}\/ 
$\SL_n (\C)$-representation.   (This is the reason why we treat first 
the outer Torelli groups.) Furthermore, it is readily seen that  
the subset $\RR (\OA_n)\subseteq (L\otimes \C)^*$
is $\SL_n (\C)$-invariant.

\begin{theorem}
\label{thm:outintro}
For $n\ge 4$, the following hold.
\begin{enumerate}
\item \label{outi0}
$\RR (\OA_n)=\{ 0\}$.
\item \label{outi1}
$\VV (\OA_n)$ is a finite set.
\item \label{outi2}
The Alexander polynomial of $\OA_n$ is a non-zero constant.
\item \label{outi3}
If $N$ is a subgroup of $\OA_n$ containing $\OA'_n$, then $b_1(N)<\infty$.
\end{enumerate}
\end{theorem}

Claim \eqref{outi0} is proved in Theorem \ref{thm:res oan}. 
We use  $\SL_n (\C)$-representation theory and 
the explicit $\SL_n (\C)$-equivariant description 
given by Pettet in \cite{Pe} for the co-restriction 
of the cup-product map $\cup_{\OA_n}$.

The key claim for the rest of the proof (given in Theorem \ref{thm:a}) 
is part \eqref{outi1}.   We know from part \eqref{outi0} that $\VV (\OA_n)$ 
is a proper, Zariski closed and $\SL_n(\Z)$-{\em invariant} subset of $\TT (L)$. 
Applying Theorem B from \cite{DP}, we conclude that $\VV(\OA_n)$ is finite. 

The above theorem is used to prove (in Theorem \ref{thm:b}) 
the analogous result for $\IA_n$.  

\begin{theorem}
\label{thm:autintro}
For $n\ge 5$, the following hold.
\begin{enumerate}
\item \label{auti1}
$\VV (\IA_n)$ is a finite set.
\item \label{auti2}
The Alexander polynomial of $\IA_n$ is a non-zero constant.
\item \label{auti3}
If $N$ is a subgroup of $\IA_n$ containing $\IA'_n$, then $b_1(N)<\infty$.
\end{enumerate}
\end{theorem}

As above, the main point is to establish the fact 
that $\dim_{\C} H_1 (\IA_n', \C)< \infty$, which is 
precisely the statement needed to finish the 
proof of Theorem \ref{thm:intro1}.
For $n\ge 3$, we have the exact sequence \eqref{eq:gexintro}. 
Due to Theorem \ref{thm:outintro}\eqref{outi3},
it is enough to check that the co-invariants of the $\IA_n'$-action 
on $H_1(F_n', \C)$ are finite-dimensional, whenever $n\ge 5$.
This in turn follows from Lemma \ref{lem:isq}, where we exploit 
the $\C {\Z^n}$-linearity of the $\IA_n$-action on 
$B(F_n)\otimes \C$ to reach the desired conclusion.

\subsection{Organization of the paper}
\label{subsec:intro org}

Roughly speaking, the paper is divided into three parts.  

In the first part (sections \ref{sect:johnson}--\ref{sect:outer johnson}), 
we study the Johnson filtrations on $\Aut(G)$ and $\Out(G)$, as well 
as the corresponding Johnson homomorphisms, $J$ and $\wJ$, 
paying particular attention to their $\A(G)$-equivariance  
properties.

In the second part (sections \ref{sect:alexinv}--\ref{sect:resalex}), 
we discuss the Alexander-type invariants $B(G)$ and $\B(G)$, as well as 
the characteristic and resonance varieties $\VV(G)$ and $\RR(G)$, 
with a focus on the varied connections between these objects and 
sundry graded Lie algebras. 

In the third part (sections \ref{sect:autfn}--\ref{sect:oian}), 
we apply the machinery developed in the first two parts to the problem 
at hand, namely, the investigation of homological finiteness properties 
in the Andreadakis--Johnson filtration of $\Aut(F_n)$. 

\section{Torelli group and Johnson homomorphism}
\label{sect:johnson}

In this section, we recall the definition of the Johnson filtration 
on the automorphism group of an arbitrary group $G$.  
Next, we recall the definition of the Johnson homomorphism, 
and prove a key equivariance property. 

\subsection{Filtered groups and graded Lie algebras}
\label{subsec:grg}
We start by reviewing some basic notions about groups 
and Lie algebras, following the exposition from Serre's 
book \cite{Se}.   

Let $G$ be a group.  Given two elements 
$x,y\in G$, write ${}^x y= xyx^{-1}$ and  
$(x,y)=xyx^{-1}y^{-1}$.  The following `Witt--Hall' 
identities then hold:
\begin{align}
\label{eq:bilg}
& (xy, z)= {}^x (y,z)\cdot (x,z),\\
\label{eq:jacobi}
& ( {}^y x, (z, y)) \cdot ({}^z y, (x, z)) \cdot ({}^x z,(y, x))= 1.
\end{align}

Given subgroups $K_1$ and $K_2$ of $G$, define $(K_1,K_2)$ 
to be the subgroup of $G$ generated by all commutators of the form 
$(x_1,x_2)$, with $x_i\in K_i$.  In particular, $G'=(G,G)$ is the 
derived subgroup of $G$. 

Now suppose we are given a decreasing filtration 
\begin{equation}
\label{eq:Phi}
G=\Phi^1 \supset \Phi^2   \supset  \cdots \supset 
\Phi^s  \supset\cdots 
\end{equation}
by subgroups of $G$ satisfying 
$(\Phi^s, \Phi^t) \subseteq \Phi^{s+t}$, for all $s,t\ge 1$.  
Clearly, each term $\Phi^s$ in the series is a normal 
subgroup of $G$, and the successive quotients, 
$\gr^{s}_{\Phi} (G)=\Phi^s/\Phi^{s+1}$, 
are abelian groups.  

Set
\begin{equation}
\label{eq:grPhi}
\gr_{\Phi} (G)  = \bigoplus_{s\ge 1} \gr^{s}_{\Phi} (G). 
\end{equation} 
Using identities \eqref{eq:bilg} and \eqref{eq:jacobi}, 
it is readily verified that $\gr_{\Phi} (G)$ has the structure 
of a graded Lie algebra, with Lie bracket $[\:,\:]$ induced 
by the group commutator
(there are no extra signs in the Lie identities!). 

The most basic example of this construction is provided 
by the lower central series, with terms 
$\Gamma^s=\Gamma^s(G)$ defined inductively by 
$\Gamma^1 =G$ and $\Gamma^{s+1} =(\Gamma^s ,G)$. 
Note that $(\Gamma^s, \Gamma^t) \subseteq \Gamma^{s+t}$.  
The resulting Lie algebra, $\gr_{\Gamma} (G)$, 
is called the {\em associated graded Lie algebra}\/ of $G$.
Note that $\gr^{1}_{\G} (G)$ coincides with the abelianization 
$G_{\ab} = G/G'$, and that $\gr_{\Gamma} (G)$ 
is generated (as a Lie algebra) by this degree $1$ piece.
Clearly,  $\gr_{\Gamma}$ is a functor from groups to graded 
Lie algebras.

Given any filtration $\{\Phi^s\}_{s\ge 1}$ as in \eqref{eq:Phi}, 
we have $\Gamma^s\subseteq \Phi^s$, for all $s\ge 1$.  
Thus, there is a canonical map  
\begin{equation}
\label{eq:iota phi}
\xymatrix{\iota_{\Phi} \colon \gr_{\Gamma} (G) \ar[r]  
& \gr_{\Phi} (G)}.
\end{equation}
Clearly, $\iota_{\Phi}$ is a morphism of graded Lie algebras. 
In degree $1$, the map $\iota_{\Phi}$ is surjective.  For higher 
degrees, though, $\iota_{\Phi}$ is neither injective nor surjective 
in general. 

To illustrate these concepts, consider the free group 
on $n$ generators, denoted $F_n$.   Work of P.~Hall, E.~Witt, 
and W.~Magnus from the 1930s (see \cite{MKS}) elucidated 
the structure of the associated graded Lie algebra of $F_n$.  
We summarize those results, as follows.  

\begin{theorem} 
\label{thm:free lie}
For all $n\ge 1$,
\begin{enumerate}
\item \label{fl1}
The group $F_n$ is residually nilpotent, i.e., 
$\bigcap_{s\ge 1} \Gamma^s (F_n) =\{1\}$. 
\item \label{fl2}
The  Lie algebra $\gr_{\Gamma}(F_n)$ is 
isomorphic to the free Lie algebra $\L_n=\Lie(\Z^n)$.
\item  \label{fl3}
For each $s\ge 1$, the group $\gr^s_{\Gamma}(F_n)$ is 
free abelian. 
\item  \label{fl4}
If $n\ge 2$, the center of $\L_n=\gr_{\Gamma}(F_n)$ is trivial,
and $\gr_{\Gamma}^s(F_n) \ne 0$, for all $s\ge 1$.
\end{enumerate}
\end{theorem}

\subsection{The Johnson filtration}
\label{subsec:jfilt}

Returning to the general situation, let $G$ be a group, and let 
$\Aut(G)$ be its automorphism group, with group operation 
$\alpha\cdot \beta:=\alpha\circ \beta$.   Note that each term 
in the lower central series of $G$ is a characteristic subgroup, 
i.e., $\alpha(\Gamma^s) = \Gamma^s$, for all automorphisms 
$\alpha\in \Aut(G)$.  

\begin{definition}
\label{def:jfilt}
The {\em Johnson filtration}\/ is the decreasing filtration 
$\{F^s\}_{s\ge 0}$ on $\Aut(G)$ given by 
\begin{equation}
\label{eq:jfilt}
F^s (\Aut(G))= \{ \alpha \in \Aut (G) 
\mid \alpha \equiv \id\: \bmod\: \Gamma^{s+1}(G)\}.
\end{equation}
\end{definition}

In other words, an automorphism $\alpha\in \Aut(G)$ 
belongs to the subgroup $F^s$ if and only if 
$\alpha(x)\cdot x^{-1} \in \Gamma^{s+1}(G)$, 
for all $x\in G$.  In fact, as noted in the proof of
Proposition 2.1 from \cite{Pa}, slightly more is true: 
if $\alpha\in F^s$ and $x\in \Gamma^t(G)$, then 
$\alpha(x)\cdot x^{-1} \in \Gamma^{s+t}(G)$. 

Since $\Gamma^{s+1}(G)$ is a characteristic subgroup, 
reducing modulo this subgroup yields a map 
$\kappa_s\colon \Aut(G) \to \Aut(G/\Gamma^{s+1}(G))$.  
It is readily seen that $\kappa_s$ is a group homomorphism, 
and $F^s = \ker (\kappa_s)$.  Hence, $F^s$ is a normal 
subgroup of $\Aut(G)$. 

A result of Kaloujnine \cite{Ka} gives that 
$(F^s, F^t) \subseteq F^{s+t}$, 
for all $s, t \ge 0$.  
Clearly, $F^{0}=\Aut(G)$.  The next term in the Johnson 
filtration is one of our key objects of study. 

\begin{definition}
\label{def:torelli}
The {\em Torelli group}\/ of $G$ is the subgroup 
$\T_G=\ker(\Aut(G)\to \Aut(G_{\ab}))$ of $\Aut(G)$ 
consisting of all automorphisms of $G$ inducing 
the identity on $G_{\ab}$.
\end{definition}

By construction, the Torelli group $F^{1}=\T_G$ 
is a normal subgroup of $F^{0}=\Aut(G)$. For 
reasons that will become apparent later on, we call the 
quotient group, $\A(G)=F^{0}/F^{1}$, the {\em symmetry group}\/ 
of $\T_G$.  Clearly, $\A(G)=\im(\Aut(G)\to \Aut(G_{\ab}))$, 
and we have a short exact sequence
\begin{equation}
\label{eq:autgseq}
\xymatrix{
1\ar[r] &\T_G \ar[r] &\Aut(G) \ar[r] & \A(G) \ar[r] & 1
}.
\end{equation}

The Torelli group inherits the Johnson filtration 
$\{F^s(\T_G)\}_{s\ge 1}$ from $\Aut(G)$. The 
corresponding graded Lie algebra, $\gr_{F} (\T_G)$, 
admits an action of $\A(G)$, defined as follows.  
Pick a homogeneous element $\bar{\alpha}\in \gr_{F}^{s} (\T_G)$, 
represented by an automorphism $\alpha\in F^s$, 
and let $\bar{\sigma} \in \A(G)$, represented by an 
automorphism $\sigma\in \Aut(G)$.  Set
\begin{equation}
\label{eq:sigma alpha}
\bar{\sigma} \cdot \bar{\alpha} = \overline{\sigma \alpha \sigma^{-1}}. 
\end{equation}
It is readily verified that this action does not depend on the choices 
made, and that it preserves the graded Lie algebra structure on 
$\gr_{F} (\T_G)$.

The Torelli group also comes endowed with its own lower 
central series filtration, $\{\Gamma^s(\T_G)\}$.  
In turn, the associated graded Lie algebra, $\gr_{\Gamma} (\T_G)$, 
admits an action of $\A(G)$, defined in a similar manner.
It is now easily checked that the morphism 
$\iota_{F} \colon \gr_{\Gamma} (\T_G) \to \gr_{F} (\T_G)$ is 
equivariant with respect to the action of $\A(G)$ on the source 
and target Lie algebras.

\subsection{The Johnson homomorphism}
\label{subsec:jh}

Given a graded Lie algebra $\g$, define the graded Lie 
algebra of positively-graded derivations of $\g$ as 
\begin{equation}
\label{eq:dplus}
\Der(\g) = \bigoplus_{s\ge 1} \Der^s(\g),
\end{equation}
where $\Der^s(\g)$ is the set of linear maps 
$\delta\colon \g^{\bullet} \to \g^{\bullet+s}$ with the property that 
$\delta [x,y]= [\delta x, y]+ [x, \delta y]$ for all $x,y\in \g$,  
and with the Lie bracket on $\Der(\g)$ given by 
$[\delta, \delta']= \delta \circ\delta'- \delta' \circ \delta$. 

The adjoint map, $\ad\colon \g \to \Der(\g)$, sends each 
$x\in \g$ to the inner derivation $\ad_x \colon \g \to \g$  
given by $\ad_x(y)=[x,y]$.   It is readily seen that 
the adjoint map is a morphism of graded Lie algebras, 
with kernel equal to the center $Z(\g)$ of $\g$.  

The following theorem/definition generalizes Johnson's 
original construction from \cite{J80}.

\begin{theorem}[\cite{Pa}]
\label{thm:johnson-hom}
For every group $G$, there is a well-defined map 
\begin{equation}
\label{eq:jhom}
\xymatrix{J \colon \gr_{F} (\T_G) \ar[r]  & \Der (\gr_{\Gamma} (G))},
\end{equation}
given on homogeneous elements $\alpha\in F^s(\T_G)$ and 
$x\in \Gamma^t(G)$ by 
\begin{equation}
\label{eq:jdef}
J (\bar{\alpha}) (\bar{x}) = \overline{\alpha(x) \cdot x^{-1}}.
\end{equation}
Moreover, $J$ is a monomorphism of graded Lie algebras. 
\end{theorem}

Next, we compare the two natural filtrations on the Torelli 
group $\T_G$:  the  lower central series filtration, $\Gamma^s(\T_G)$, 
and the Johnson filtration, $F^s(\T_G)$.  From the discussion 
in \S\ref{subsec:grg}, we know that there always exist inclusions 
$\Gamma^s (\T_G) \inj F^s (\T_G)$ inducing homomorphisms 
$\iota_F\colon \gr^s_{\G}(\T_G) \to \gr^s_F(\T_G)$, for all $s\ge 1$. 
In general, the two filtrations are not equal.  Nevertheless, the 
next result provides a necessary and sufficient condition (in terms 
of the Johnson homomorphism) under which the two filtrations 
coincide, up to a certain degree. 

\begin{theorem}
\label{thm:two filt}
Let $G$ be a group. For each $q\ge 1$, the following are equivalent:
\begin{enumerate} 
\item \label{gf1}
 $J\circ \iota_{F}\colon \gr^s_{\G}(\T_G) \to \Der^s (\gr_{\Gamma} (G))$ 
is injective, for all $s\le q$. 
\item \label{gf2} $\Gamma^s(\T_G)= F^s(\T_G)$, for all $s\le q+1$. 
\end{enumerate}
\end{theorem}

\begin{proof}
To prove \eqref{gf1} $\Rightarrow$ \eqref{gf2}, we use induction 
on $s$, for $s\le q+1$.  For $s=1$, we have $\Gamma^1(\T_G)=F^1(\T_G)=\T_G$. 
Assume now that $\Gamma^s (\T_G) =F^s (\T_G)$, for some $s\le q$, 
and consider the following commuting diagram:
\begin{equation}
\label{eq:tgprime}
\xymatrix{
1\ar[r] & \Gamma^{s+1}(\T_G) \ar[r] \ar@{^{(}->}[d] 
& \Gamma^{s}(\T_G) \ar[r] \ar^{=}[d]
& \gr^s_{\G}(\T_G) \ar[r]  \ar@{^{(}->}^{\iota_F}[d]&  1\,\phantom{.}\\
1\ar[r] & F^{s+1}(\T_G)  \ar[r] &  F^{s}(\T_G) \ar[r] 
& \gr^s_{F}(\T_G) \ar[r] &  1\,.
}
\end{equation}

The assumption that 
$J\circ \iota_{F}$ is injective in degree $s$ implies that 
$\iota_F$ is injective in degree $s$.   From diagram \eqref{eq:tgprime}, 
we conclude that the inclusion $\Gamma^{s+1}(\T_G) \inj F^{s+1}(\T_G)$ 
is an equality, thereby finishing the induction step.

We now prove \eqref{gf2} $\Rightarrow$ \eqref{gf1}.  Looking 
again at diagram \eqref{eq:tgprime}, the assumption that 
$\Gamma^s(\T_G)= F^s(\T_G)$ for all $s\le q+1$ implies that 
$\iota_F$ is an isomorphism in degrees $s\le q$. By 
Theorem \ref{thm:johnson-hom}, then, the map 
$J\circ \iota_{F}$ is injective in degrees $s\le q$.
\end{proof}

\subsection{Equivariance of $J$}
\label{subsec:equiv johnson}
Recall that the symmetry group $\A(G)=\im(\Aut(G)\to \Aut(G_{\ab}))$ 
acts on the graded Lie algebra $\gr_F(\T_G)$ via the action 
given by \eqref{eq:sigma alpha}.  In turns out that this group 
also acts on $\Der (\gr_{\Gamma} (G))$, in a manner we 
now describe. 

First, let us define an action of $\A(G)$ on the graded 
Lie algebra $\g=\gr_{\Gamma} (G)$.  Pick an element 
$\bar{x}\in \gr_{\Gamma}^s (G)$, represented 
by $x\in \Gamma^s(G)$, and let $\bar{\sigma} \in \A(G)$, 
represented by $\sigma\in \Aut(G)$.  Set
\begin{equation}
\label{eq:sigma x}
\bar{\sigma} \cdot \bar{x} = \overline{\sigma(x)}. 
\end{equation}
It is readily verified that this action is well-defined, and that it 
preserves both the Lie bracket and the grading on $\g$.  The 
action of $\A(G)$ on $\g$ extends to $\Der(\g)$, by setting
\begin{equation}
\label{eq:der act}
\bar{\sigma} \cdot \delta = \bar{\sigma} \delta \bar{\sigma}^{-1},
\end{equation}
for any $\bar{\sigma}\in \A(G)$ and $\delta\in \Der(\g)$.  
Again, it is routine to verify that this action is well-defined, 
and preserves the graded Lie algebra structure on $\Der(\g)$.

The next proposition sharpens Theorem \ref{thm:johnson-hom}.

\begin{prop}
\label{prop:johnson equiv}
The Johnson homomorphism 
$J \colon \gr_{F} (\T_G) \to \Der (\gr_{\Gamma} (G))$ is equivariant 
with respect to the actions of $\A(G)$ on source and target defined above.
\end{prop}

\begin{proof}
Let $\bar{\sigma}\in \A(G)$, represented by $\sigma \in \Aut (G)$, 
and $\bar{\alpha}\in \gr_F^s (\T_G)$, represented by 
$\alpha\in F^s$.  We need to verify  that
$J(\overline{\sigma\alpha\sigma^{-1}})= 
\bar{\sigma} J(\bar{\alpha}) \bar{\sigma}^{-1}$.
Taking an arbitrary $x\in \Gamma^t$ and evaluating 
on $\bar{x}\in \gr_{\Gamma}^t (G)$, we get
\[
J(\overline{\sigma\alpha\sigma^{-1}})(\bar{x}) = 
\overline{\sigma\alpha\sigma^{-1}(x) \cdot x^{-1}} =
\overline{\sigma (\alpha(\sigma^{-1} (x)) \cdot 
\sigma^{-1} (x)^{-1})} = \bar{\sigma} J(\bar{\alpha}) 
\bar{\sigma}^{-1} (\bar{x}),
\]
and this finishes the proof.
\end{proof}

\section{The outer Torelli group}
\label{sect:outer torelli}

In this section, we study the quotient of the Torelli group 
by the subgroup of inner automorphisms. 

\subsection{The outer Torelli group}
\label{subsec:out torelli}

Let $\Ad\colon G\to \Aut(G)$ be the adjoint map, sending 
an element $x\in G$ to the inner automorphism $\Ad_x\colon G\to G$, 
$y\mapsto xyx^{-1}$.  Clearly, $\Ad_{xy}=\Ad_x \Ad_y$, and so 
$\Ad$ is a homomorphism; its kernel is the center of $G$, while 
its image is the group of inner automorphisms, $\Inn(G)$.  

The adjoint homomorphism is equivariant with respect to 
the action of $\Aut(G)$ on source (by evaluation), and target 
(by conjugation); that is, 
\begin{equation}
\label{eq:adequiv}
\Ad_{\alpha(x)}=\alpha\circ \Ad_x \circ \alpha^{-1},
\end{equation}
for all $\alpha\in \Aut (G)$ and $x\in G$.  Consequently, $\Inn(G)$ 
is a normal subgroup of $\Aut(G)$. Let $\Out(G)$ be the factor group.  
We then have an exact sequence,
\begin{equation}
\label{eq:outg}
\xymatrix{
1\ar[r] &\Inn(G) \ar[r] &\Aut(G) \ar^{\pi}[r] & \Out(G) \ar[r] & 1
}.
\end{equation}

The Johnson filtration $\{F^s\}_{s\ge 0}$ on $\Aut(G)$ 
yields a filtration $\{\wF^s\}_{s\ge 0}$  on $\Out(G)$, 
by setting $\wF^s:=\pi(F^s)$.  Clearly, 
$(\wF^s, \wF^t) \subseteq \wF^{s+t}$, 
for all $s, t \ge 0$.  

\begin{definition}
\label{def:outer torelli}
The {\em outer Torelli group}\/ of $G$ is the subgroup 
$\wT_G=\wF^{1}$ of $\Out(G)$, consisting of all outer 
automorphisms inducing the identity on $G_{\ab}$.
\end{definition}

By construction, the outer Torelli group $\wF^{1}=\wT_G$ 
is a normal subgroup of $\wF^{0}=\Out(G)$.  The quotient group,  
$\wF^{0}/\wF^{1}=\im(\Out(G)\to \Aut(G_{\ab}))$, is isomorphic 
to $\A(G)$, and we have a short exact sequence
\begin{equation}
\label{eq:outgseq}
\xymatrix{
1\ar[r] &\wT_G \ar[r] &\Out(G) \ar[r] & \A(G) \ar[r] & 1
}.
\end{equation}
It is readily seen that $\Inn(G)$ is a normal subgroup of $\T_G$, and   
the quotient group, $\T_G/\Inn(G)$, is isomorphic to $\wT_G$.

The outer Torelli group inherits the Johnson filtration 
$\{\wF^s(\wT_G)\}_{s\ge 1}$ from $\Out(G)$. The 
corresponding graded Lie algebra, $\gr_{\wF} (\wT_G)$, 
admits an action of $\A(G)$, defined as in \eqref{eq:sigma alpha}. 
Evidently, the canonical projection $\pi\colon \Aut(G)\surj \Out(G)$ 
induces an epimorphism of graded Lie algebras, 
$\bar{\pi}\colon \gr_{F} (\T_G)\surj \gr_{\wF} (\wT_G)$, 
which is $\A(G)$-equivariant with respect to the given actions. 

The conjugation action of $\Out (G)$ on $\wT_G$ 
induces an action of $\A(G)$ on 
$\gr_{\Gamma}(\wT_G)$, preserving the graded Lie algebra 
structure.  Moreover, both 
$\iota_{\wF}\colon \gr_{\Gamma}(\wT_G) \to \gr_{\wF} (\wT_G)$ and
$\gr_{\Gamma} (\pi)\colon \gr_{\Gamma}(\T_G) \surj \gr_{\Gamma}(\wT_G)$ 
are $\A(G)$-equivariant, and the following diagram commutes in the 
category of graded Lie algebras endowed with compatible 
$\A(G)$-module structures:
\begin{equation}
\label{eq:iotapi cd}
\xymatrixcolsep{36pt}
\xymatrix{
\gr_{\G} (\T_G) \ar^(.5){\iota_F}[r]  \ar@{->>}^(.46){\gr_{\G}(\pi)}[d]
&\gr_{F}(\T_G) \ar@{->>}^{\bar{\pi}}[d]\\
\gr_{\G} (\wT_G) \ar^(.5){\iota_{\wF}}[r]  
&\gr_{\wF}(\wT_G) \,.
}
\end{equation}

\subsection{Adjoint homomorphism and Johnson filtration}
\label{subsec:ses}

The image of the homomorphism $\Ad\colon G\to \Aut(G)$ is 
the inner automorphism group $\Inn(G)$, which is contained 
in the Torelli group $\T_G$.  The co-restriction 
$\Ad\colon G\to \T_G$ has good compatibility properties 
with respect to the filtrations $\Gamma^s=\Gamma^s(G)$ 
on source and $\Gamma^s(\T_G)$ and $F^s(\T_G)$ on target.  
More precisely, we have the following lemma. In \eqref{ad2} below,
we will be mainly interested in the particular case when $A$ is 
$\Ad\colon G\to \T_G$.

\begin{lemma}
\label{lem:adx}
Let $G$ be a group, and let $A\colon G\to \T$ be a group 
homomorphism.
\begin{enumerate}
\item \label{ad1}
If $x\in \Gamma^s$, then $\Ad_x \in F^s(\T_G)$.

\item \label{ad2}
Suppose $\gr_{\G} (A)\colon \gr_{\G} (G) \to \gr_{\G} (\T)$ is 
injective.  If $A(x) \in \Gamma^s(\T)$, then $x\in \Gamma^s (G)$.

\item \label{ad3}
Suppose $\gr_{\G}(G)$ has trivial center. If 
$\Ad_x \in F^s(\T_G)$, then $x\in \Gamma^s$.
\end{enumerate}
\end{lemma}

\begin{proof}
\eqref{ad1} 
If $x$ belongs to $\Gamma^s(G)$, then 
clearly  $\Ad_x(y) \equiv y \pmod{\Gamma^{s+1}}$, 
and thus $\Ad_x $ belongs to $F^s(\T_G)$.

\eqref{ad2} 
Let $x\in G$, and assume $A(x) \in \Gamma^s(\T)$.  
We prove by induction on $r$ that $x\in \Gamma^r (G)$, for 
all $r\le s$.  For $r=1$, the conclusion is tautologically true. 
So assume $x\in \Gamma^{r}(G)$, for some $r<s$. 
Let $\bar{x}$ be the class of $x$ in $\gr^r_{\G} (G)$.
Then $\gr_{\G} (A) (\bar{x})=\overline{A(x)}=0$ in $\gr^r_{\G} (\T)$, 
since $A(x) \in \G^s(\T)$ and $r<s$.  From our injectivity 
assumption, we obtain $\bar{x}=0$, and thus  $x\in \Gamma^{r+1}(G)$.  
This finishes the induction step, and thus proves the claim.

\eqref{ad3} 
Again, we prove by induction on $r$ that $x\in \Gamma^r$, 
for all $r\le s$.  For $r=1$, the conclusion is tautologically 
true.  So assume $x\in \Gamma^{r}$, for some $r<s$. 
By definition of the Johnson filtration, $\Ad_x\in F^s(\T_G)$ 
means that ${}^x y \equiv y \pmod{\Gamma^{s+1}}$, or 
equivalently, $(x,y)\in \Gamma^{s+1}$, for all $y\in G$. 

Now, since $r<s$, we must have $(x,y)\in \Gamma^{r+2}$, 
for all $y\in G$.  Denoting by $\bar{x}$ the class of $x$ in 
$\gr_{\G}^r(G)$, we find that $[\bar{x},\bar{y}] = 0$, for all 
$\bar{y}\in \gr^{1}_{\G}(G)$.   Using the fact that 
$\gr_{\Gamma}$ is generated in degree $1$, we conclude 
that $\bar{x}$ belongs to the center of $\gr_{\Gamma} (G)$.   
By hypothesis, $Z(\gr_{\Gamma} (G))=0$;  therefore, $\bar{x}=0$, 
that is, $x\in \Gamma^{r+1}$.  This finishes the induction 
step, and thus proves the claim.
\end{proof}

\begin{prop}
\label{prop:exone}
Suppose the group $G$ is residually nilpotent, i.e., 
$\bigcap_{s\ge 1} \Gamma^s (G) =\{1\}$, 
and the Lie algebra $\gr_{\Gamma} (G)$ has trivial center. 
Then, for each $s\ge 1$, we have an exact sequence 
\begin{equation}
\label{eq:ex1}
\xymatrix{
1\ar[r] &\Gamma^s (G) \ar^(.45){\Ad}[r] & F^s (\T_G)  
\ar^{\pi}[r] &\wF^s (\wT_G)  \ar[r] & 1
}.
\end{equation}
\end{prop}

\begin{proof}
The assumptions on $G$ and $\gr_{\Gamma}(G)$ imply that 
$G$ has trivial center.  Consequently, the map $\Ad$ is injective. 
By definition of $\widetilde{F}^s$, the map $\pi$ is surjective, 
and $\pi\circ \Ad=0$. 

Now let $\alpha\in F^s$, and assume $\pi(\alpha)=0$, that is, 
$\alpha=\Ad_x$, for some $x\in G$. By Lemma \ref{lem:adx}\eqref{ad3}, 
we must have $x\in \G^s$. Hence, $\alpha\in \im(\Ad)$.
This shows that $\ker (\pi)\subseteq \im (\Ad)$, thereby 
establishing the exactness of \eqref{eq:ex1}.
\end{proof}

\subsection{Passing to the associated graded}
\label{subsec:assoc gr}
Returning to the general situation, consider an arbitrary group $G$.  
The homomorphism $\Ad\colon G \to \T_G$ induces a morphism 
$\gr_{\G}(\Ad)\colon \gr_{\G}(G) \to \gr_{\G}(\T_G)$ between the 
respective associated graded Lie algebras.  Making use of 
Lemma \ref{lem:adx}\eqref{ad1}, we may also define a map
\begin{equation}
\label{eq:barad}
\xymatrix{\overline{\Ad} \colon \gr_{\Gamma} (G) \ar[r]& \gr_{F} (\T_G)},
\end{equation}
by sending each $\bar{x}\in \gr^s_{\Gamma}(G)$ to 
$\overline{\Ad_{x}} \in  \gr^s_{F} (\T_G)$.   

\begin{lemma}
\label{lem:barad}
Let $G$ be a group.
\begin{enumerate}
\item \label{bar1}
The map $\overline{\Ad}\colon \gr_{\G}(G) \rightarrow \gr_F (\T_G)$ 
is an $\A(G)$-equivariant morphism of graded Lie algebras.

\item \label{bar2}
The map $\gr_{\G}(\Ad)\colon \gr_{\G}(G) \rightarrow \gr_{\G} (\T_G)$ 
is an $\A(G)$-equivariant morphism of graded Lie algebras.  
Moreover, $\iota_F \circ \gr_{\G}(\Ad)= \overline{\Ad}$. 

\item \label{bar3}
Let $J\colon   \gr_{F}(\T_G) \to \Der ( \gr_{\G}(G))$ 
be the Johnson homomorphism, and let 
$\ad\colon  \gr_{\G}(G) \to \Der ( \gr_{\G}(G))$ be 
the adjoint map.  Then $J\circ \overline{\Ad}=\ad$.  
\end{enumerate}
\end{lemma}

\begin{proof}
\eqref{bar1}
Let $\bar{x}\in \gr_{\Gamma}^s (G)$, represented by $x\in \G^s(G)$,  
and let $\bar{y}\in \gr_{\Gamma}^t (G)$, represented by $y\in \G^t(G)$.  
We then have:
\begin{equation}
\label{eq:ch0}
\overline{\Ad}([\bar{x},\bar{y}]) = 
\overline{\Ad}\big(\overline{(x,y)}\big) = 
\overline{\Ad_{(x,y)}} = \overline{(\Ad_{x},\Ad_{y})} = 
[\overline{\Ad_x},\overline{\Ad_y}] =
[\overline{\Ad}(\bar{x}),\overline{\Ad}(\bar{y})],
\end{equation}
thereby showing that $\overline{\Ad}$ respects Lie brackets. 

Next, let $\bar{\sigma}\in \A(G)$, represented by $\sigma \in \Aut (G)$.  
Using \eqref{eq:adequiv}, we compute
\begin{equation}
\label{eq:ch1}
\overline{\Ad}(\bar\sigma\cdot\bar{x}) = \overline{\Ad_{\sigma(x)}} = 
\overline{\sigma\circ \Ad_x \circ \sigma^{-1}}=
\bar{\sigma}\cdot \overline{\Ad}(\bar{x})\cdot \bar{\sigma}^{-1},
\end{equation}
thereby showing that $\overline{\Ad}$ is $\A(G)$-equivariant.  

\eqref{bar2}
The  $\A(G)$-equivariance of the morphism $\gr_{\G}(\Ad)$ 
is proved exactly as above, while the equality 
$\iota_F \circ \gr_{\G}(\Ad)= \overline{\Ad}$ follows  
directly from the definitions. 

\eqref{bar3}
Finally, 
\begin{equation}
\label{eq:ch2}
J\circ \overline{\Ad}(\bar{x}) (\bar{y})= 
J(\overline{\Ad_x})(\bar{y}) = \overline{\Ad_x(y)\cdot y^{-1}} = 
\overline{(x,y)} = [\bar{x},\bar{y}] = \ad (\bar{x}) (\bar{y}),
\end{equation}
thereby verifying the last assertion. 
\end{proof}

\begin{theorem}
\label{thm:piad}
Suppose $Z(\gr_{\Gamma} (G))=0$.  Then
\begin{equation}
\label{eq:exact}
\xymatrixcolsep{36pt}
\xymatrix{
0\ar[r] &\gr_{\Gamma} (G) \ar^(.48){\gr_{\G}(\Ad)}[r] \ar^{=}[d] 
& \gr_{\G} (\T_G) \ar^{\gr_{\G}(\pi)}[r] \ar^{\iota_F}[d] 
&\gr_{\G} (\wT_G) \ar^{\iota_{\wF}}[d] \ar[r] & 0\\
0\ar[r] &\gr_{\Gamma} (G) \ar^(.48){\overline{\Ad}}[r] 
& \gr_{F} (\T_G) \ar^{\bar\pi}[r] &\gr_{\wF} (\wT_G) \ar[r] & 0
}
\end{equation}
is a commuting diagram of graded Lie algebras with 
compatible $\A(G)$-module structures, and has exact rows.
\end{theorem}

\begin{proof}
We already know that all the maps in this diagram are 
morphisms of graded Lie algebras with compatible 
$\A(G)$-module structures.  Moreover, the left square 
commutes by Lemma \ref{lem:barad}\eqref{bar2}, 
while the right square is just the commuting square from 
diagram \eqref{eq:iotapi cd}.  By construction, the morphisms 
$\gr_{\G}(\pi)$ and $\bar\pi$ are surjective, while 
$\gr_{\G}(\pi) \circ \gr_{\G}(\Ad)=0$ and $\bar{\pi}\circ \overline{\Ad}=0$.
Thus, we are left with verifying four assertions. 

(1) $\ker(\overline{\Ad})=0$. 
By Lemma \ref{lem:barad}\eqref{bar3}, we have 
$J\circ \overline{\Ad}=\ad$. By the hypothesis on the center of 
$\gr_{\Gamma} (G)$, the map $\ad\colon  \gr_{\G}(G) \to 
\Der (\gr_{\G}(G))$ is injective. Thus, $\overline{\Ad}$ is injective. 

(2) $\ker (\gr_{\G}(\Ad))=0$. By commutativity of the left square, 
$\gr_{\G}(\Ad)$ is also injective.  

(3) $\ker(\gr_{\G}(\pi))\subseteq \im(\gr_{\G}(\Ad))$.  
Let $\bar{\alpha}\in \gr_{\G}^s(\T_G)$, represented by an automorphism 
$\alpha\in \G^{s}(\T_G)$, and assume $\bar{\alpha}$ belongs to the 
kernel of $\gr_{\G}(\pi)$.  
This is equivalent to $\pi(\alpha)\in \G^{s+1}(\wT_G)$, that is to say, 
$\pi(\alpha) = \pi(\beta)$, for some $\beta\in \G^{s+1}(\T_G)$. 
We may rewrite this condition as $\alpha=\beta \Ad_x$, for some $x\in G$. 
In particular, $\Ad_x \in \Gamma^s(\T_G)$.
By Lemma \ref{lem:adx}\eqref{ad2}, we must have $x\in \G^s$.  
Hence, $\bar{\alpha}= \overline{\Ad_x}= \gr_{\G}(\Ad)(\bar{x})$ 
belongs to the image of $\gr_{\G}(\Ad)$.

(4) $\ker(\bar\pi)\subseteq \im(\overline{\Ad})$. 
Let $\bar{\alpha}\in \gr_{F}^s(\T_G)$, represented by an automorphism 
$\alpha\in F^{s}$, and assume $\bar{\alpha}\in \ker(\bar\pi)$.  
This is equivalent to $\pi(\alpha)\in \wF^{s+1}$, that is to say, 
$\pi(\alpha) = \pi(\beta)$, for some $\beta\in F^{s+1}$.  We may 
rewrite this condition as $\alpha=\beta \Ad_x$, for some $x\in G$. 
In particular, $\Ad_x \in F^s$.  By Lemma \ref{lem:adx}\eqref{ad3}, 
we must have $x\in \G^s$. Hence, 
$\bar{\alpha}= \overline{\Ad_x}= \overline{\Ad} (\bar{x})$, 
and we are done.
\end{proof}

\section{The outer Johnson homomorphism}
\label{sect:outer johnson}

In this section we develop an outer version of the Johnson 
homomorphism, and we use both homomorphisms to compare 
the natural filtrations on the outer Torelli group.

\subsection{Outer derivations}
\label{subsec:out der}

Let $\g$ be a graded Lie algebra, and 
let $\ad\colon \g \to \Der(\g)$ be the adjoint map.   
It is readily seen that the image of this map is a Lie ideal 
in $\Der(\g)$.  Define the Lie algebra of outer, positively 
graded derivations, $\widetilde{\Der}(\g)$, to be the 
quotient Lie algebra by this ideal. When $Z(\g)=0$, 
we obtain a short exact sequence of graded Lie algebras, 
\begin{equation}
\label{eq:outder}
\xymatrix{
0\ar[r] &\g \ar^(.35){\ad}[r] &\Der(\g) \ar^{q}[r] 
&\widetilde{\Der}(\g) \ar[r] & 0
}.
\end{equation}

Now let $G$ be a group. 
Recall that both the associated graded Lie algebra $\gr_{\G}(G)$ 
and its Lie algebra of positively-graded derivations, $\Der (\gr_{\G}(G))$, 
come equipped with naturally defined actions of the group 
$\A(G)=\im(\Aut(G)\to \Aut(G_{\ab}))$. 

\begin{lemma}
\label{lem:eqder}
The adjoint map $\ad\colon \gr_{\Gamma} (G) \to 
\Der (\gr_{\Gamma} (G))$ is $\A(G)$-equivariant.
\end{lemma}

\begin{proof}
We must show that $\ad (\bar{\sigma}\cdot \bar{x}) = 
\bar{\sigma}\cdot \ad(\bar{x})\cdot \bar{\sigma}^{-1}$,
for all $\bar{\sigma}\in \A(G)$ and all $\bar{x}\in \gr_{\G}(G)$. 
Evaluating on an element $\bar{y}\in \gr_{\G}(G)$, we find that 
\[
\ad(\bar{\sigma}\cdot \bar{x})(\bar{y}) = 
[\bar{\sigma}\cdot \bar{x}, \bar{y}] =
\bar{\sigma} ( [\bar{x}, \bar{\sigma}^{-1}(\bar{y})] ) = 
\bar{\sigma}\ad(\bar{x}) \bar{\sigma}^{-1}(\bar{y}),
\]
since $\bar{\sigma}$ preserves the Lie bracket.
\end{proof}

Using this lemma and the discussion preceding it, 
we obtain the following immediate corollary.

\begin{corollary}
\label{cor:eqder}
Suppose $Z(\gr_{\Gamma} (G))=0$.  Then 
\begin{equation}
\label{eq:outder grg}
\xymatrix{
0\ar[r] &\gr_{\Gamma} (G) \ar^(.4){\ad}[r] &\Der(\gr_{\Gamma} (G)) \ar^{q}[r] 
&\widetilde{\Der}(\gr_{\Gamma} (G)) \ar[r] & 0
}
\end{equation}
is an exact sequence of graded 
Lie algebras with a compatible $\A(G)$-module structure.
\end{corollary}

\subsection{The outer Johnson homomorphism}
\label{subsec:ojh}

Recall from Theorem \ref{thm:johnson-hom}  the Johnson homomorphism, 
$J \colon \gr_{F} (\T_G) \inj \Der (\gr_{\Gamma} (G))$, defined on 
homogeneous elements by 
$J (\bar{\alpha}) (\bar{x}) = \overline{\alpha(x) \cdot x^{-1}}$.
The next theorem/definition provides an ``outer" version 
of this homomorphism. 

\begin{theorem}
\label{thm:outer johnson}
Suppose $Z(\gr_{\Gamma} (G))=0$.  Then the Johnson 
homomorphism induces a monomorphism of graded Lie algebras, 
\begin{equation}
\label{eq:outer jhom}
\xymatrix{\widetilde{J} \colon \gr_{\wF} (\wT_G) \ar[r]  
& \widetilde{\Der} (\gr_{\Gamma} (G))},
\end{equation}
which is equivariant with respect to the naturally defined 
actions of $\A(G)$ on source and target.
\end{theorem}

\begin{proof}

Consider the following diagram, in the category of graded Lie 
algebras endowed with a compatible $\A(G)$-module structure:
\begin{equation}
\label{eq:jcd}
\xymatrixcolsep{36pt}
\xymatrix{
\gr_{\Gamma} (G)  \ar^{=}[r]  \ar@{^{(}->}^(.46){\overline{\Ad}}[d]
& \gr_{\Gamma} (G) \ar@{^{(}->}^(.46){\ad}[d]\\
\gr_{F} (\T_G) \ar^(.4){J}[r]  \ar@{->>}^(.46){\bar{\pi}}[d]
& \Der (\gr_{\Gamma} (G)) \ar@{->>}^{q}[d]\\
\gr_{\wF} (\wT_G) \ar@{..>}^(.4){\widetilde{J}}[r]  
& \widetilde{\Der} (\gr_{\Gamma} (G))
}
\end{equation}

By Lemma \ref{lem:barad}, Part \eqref{bar3}, the top square in this 
diagram commutes.  In view of the hypothesis that $\gr_{\G}(G)$ is 
centerless, Corollary \ref{cor:eqder} shows that the right-hand column 
in \eqref{eq:jcd} is exact. By Theorem \ref{thm:piad}, the 
left-hand column is also exact.  These facts together imply 
the existence and uniqueness of the dotted arrow $\widetilde{J}$ 
having the desired properties, and making the bottom square commute.  
\end{proof}

\subsection{Comparing two filtrations}
\label{subsec:two outer filt}
We conclude this section with a comparison between the 
two natural filtrations on the outer Torelli group $\wT_G$:  
the  lower central series filtration, $\Gamma^s(\wT_G)$, 
and the Johnson filtration, $\wF^s(\wT_G)$.  We start 
with a comparison of the two Johnson homomorphisms.

\begin{lemma}
\label{lem:filt tg}
Suppose $Z(\gr_{\Gamma} (G))=0$. For a fixed $s\ge 1$, 
the following statements are equivalent:
\begin{enumerate} 
\item \label{ft1}
$J\circ \iota_{F}\colon \gr^s_{\G}(\T_G) \to 
\Der^s (\gr_{\Gamma} (G))$ is injective (respectively, surjective). 
\item \label{ft2} $\widetilde{J}\circ \iota_{\wF}\colon \gr^s_{\G}(\wT_G) \to 
\widetilde{\Der}^s (\gr_{\Gamma} (G))$ is injective (respectively, surjective).
\end{enumerate}
\end{lemma}

\begin{proof}
Chase diagram \eqref{eq:exact} from Theorem \ref{thm:piad} and 
diagram \eqref{eq:jcd} from Theorem \ref{thm:outer johnson}. 
\end{proof}

As we know, there always exist inclusions 
$\Gamma^s (\wT_G) \inj \wF^s (\wT_G)$ inducing homomorphisms 
$\iota_{\wF}\colon \gr^s_{\G}(\wT_G) \to \gr^s_{\wF}(\wT_G)$, 
for all $s\ge 1$.  In general, the two natural filtrations on 
$\wT_G$ are not equal.  Nevertheless, the next result reformulates 
the equality of the two filtrations, up to a 
certain degree, in terms of the Johnson homomorphisms. 

\begin{theorem}
\label{thm:two outfilt}
Suppose $Z(\gr_{\Gamma} (G))=0$. For each $q\ge 1$, 
the following statements are equivalent:
\begin{enumerate} 
\item \label{of1}
$J\circ \iota_{F}\colon \gr^s_{\G}(\T_G) \to 
\Der^s (\gr_{\Gamma} (G))$ is injective, for all $s\le q$. 

\item \label{of2} 
$\widetilde{J}\circ \iota_{\wF}\colon \gr^s_{\G}(\wT_G) \to 
\widetilde{\Der}^s (\gr_{\Gamma} (G))$ is injective, for all $s\le q$.

\item \label{of3} $\Gamma^s(\T_G)= F^s(\T_G)$, for all $s\le q+1$.

\item \label{of4} $\Gamma^s(\wT_G)= \wF^s(\wT_G)$, for all $s\le q+1$.

\end{enumerate}
\end{theorem}

\begin{proof}
\eqref{of1} $\same$ \eqref{of2}.  This was established 
in Lemma \ref{lem:filt tg}. 

\eqref{of1} $\same$ \eqref{of3}. This fact was proved in
Theorem \ref{thm:two filt}.  

\eqref{of2} $\same$ \eqref{of4}.  
The claim follows by the same argument used in proving 
Theorem \ref{thm:two filt}.  
\end{proof}

\begin{corollary}
\label{cor:exot}
Let $G$ be a residually nilpotent group.  Suppose $\gr_{\Gamma} (G)$ 
is centerless and $J\circ \iota_{F}\colon \gr^s_{\G}(\T_G) \to 
\Der^s (\gr_{\Gamma} (G))$ is injective, for all $s\le q$.   Then
the sequence 
\begin{equation}
\label{eq:ads}
\xymatrix{
1\ar[r] &\G^s (G) \ar^(.45){\Ad}[r] & \G^s (\T_G)
\ar^{\pi}[r] & \G^s (\wT_G) \ar[r] & 1
}
\end{equation}
is exact, for all $s\le q+1$.
\end{corollary}

\begin{proof}
Consider the exact sequence \eqref{eq:ex1} 
from Proposition \ref{prop:exone}.  Using 
Theorem \ref{thm:two outfilt} 
to replace $F^s(\T_G)$ by $\Gamma^s(\T_G)$ 
and $\wF^s(\wT_G)$ by $\Gamma^s(\wT_G)$ 
ends the proof.
\end{proof}

It is worth highlighting separately the case $q=1$   
of the above results. 

\begin{corollary}
\label{cor:outer johnson2}
Suppose $G$ is residually nilpotent, $\gr_{\Gamma} (G)$ is centerless, 
and $J \circ \iota_{F}\colon \gr^1_{\G}(\T_G) \to \Der^1 (\gr_{\Gamma} (G))$ 
is injective. Then $F^2(\T_G)= \T_G'$, $\wF^2(\wT_G)= \wT_G'$, 
and we have an exact sequence
\begin{equation}
\label{eq:adprime}
\xymatrix{
1\ar[r] &G' \ar^(.45){\Ad}[r] & \T_G' 
\ar^{\pi}[r] &\wT_G' \ar[r] & 1
}.
\end{equation}
\end{corollary}

\section{The Alexander invariant and its friends}
\label{sect:alexinv}

We now turn to the Alexander invariant of a group. We discuss  
the action of the Torelli group on this module, and some of its 
connections with associated graded algebras.

\subsection{The abelianization of the derived subgroup}
\label{subsec:alex inv}

Let $G$ be a group.  Recall $G'=(G,G)$ is  
the derived subgroup, and $G_{\ab}=G/G'$ is the  
maximal abelian quotient of $G$.  Similarly, $G''=(G',G')$ 
is the second derived subgroup, and $G/G''$ is the maximal 
metabelian quotient.  The abelianization map 
$\ab\colon G\surj G_{\ab}$ factors through $G/G''$, 
and so we get an exact sequence
\begin{equation} 
\label{eq:ggg}
\xymatrix{0 \ar[r]& G'/G'' \ar[r]& 
G/G'' \ar[r]& G_{\ab} \ar[r]& 0}. 
\end{equation}

Conjugation in $G/G''$ naturally makes the abelian group 
$G'/G''$ into a module over the group ring $R=\Z{G_{\ab}}$.   
We call this module, 
\begin{equation} 
\label{eq:alex inv}
B(G) := G'/G'' = H_1(G',\Z),
\end{equation}
the {\em Alexander invariant}\/ of $G$. By definition, 
the $R$-module structure on $B(G)$ is given by 
\begin{equation} 
\label{eq:ai action}
\bar{h}\cdot \bar{g}= \overline{hgh^{-1}},
\end{equation}
for $\bar{h}\in G_{\ab}=G/G'$, represented by $h\in G$, and 
$\bar{g}\in B(G)=G'/G''$, represented by $g\in G'$. Since 
$G'=\Gamma^2(G)$, it follows that $G'' \subseteq \Gamma^4(G)
\subseteq \Gamma^3(G)$;  
in particular, $\gr_{\Gamma}^2 (G)=\Gamma^2(G)/\Gamma^3(G)$ 
is a (generally proper) quotient of $B(G)$.

\begin{example}
\label{ex:bfree}
Let $F_n$ be the free group on generators $x_1,\dots, x_n$. 
Identify the ring $R=\Z{\Z^n}$ with the ring of Laurent 
polynomials in the variables $\bar{x}_1,\dots ,\bar{x}_n$. 
A standard argument based on induction on word length 
shows that $B(F_n)$ is generated as an $R$-module 
by all elements of the form $\overline{(x_i,x_j)}$, 
with $1\le i<j\le n$, subject to the relations 
\begin{equation}
\label{eq:bf rels}
(\bar{x}_i-1) \cdot \overline{(x_j,x_k)} - 
(\bar{x}_j-1) \cdot \overline{(x_i,x_k)} + 
(\bar{x}_k-1) \cdot \overline{(x_i,x_j)}=0,
\end{equation}
for all $1\le i<j<k \le n$.  (The fact that these relations hold in 
$F_n'/F_n''$ is a direct consequence of the Witt--Hall identity 
\eqref{eq:jacobi}; it is an exercise to show that no other relations hold.)
\end{example}

\begin{remark}
\label{rem:bpres}
More generally, if $G$ is a finitely generated group, one can 
write down an explicit presentation for the $\Z{G_{\ab}}$-module $B(G)$, 
starting from a presentation for the group.  The procedure involves 
computing the (abelianized) Fox derivatives of the relators, and 
solving a matrix lifting problem. For details, we refer to \cite{Ms}, 
and also \cite[\S8.3]{PS-imrn}.
\end{remark}

The module $B(G)$ is said to have {\em trivial $G_{\ab}$-action}\/ 
if $\bar{h}\cdot \bar{g}=\bar{g}$, for all $\bar{h}\in G/G'$ and 
$\bar{g}\in G'/G''$.  This happens precisely when $(h,g)\in G''$, 
for all $h\in G$ and $g\in G'$; that is to say, $(G,G')\subseteq G''$. 

In the sequel, we will also consider Alexander invariants 
with field coefficients.  Let $\k$ be a field, and view the 
$\k$-vector space $B_{\k}(G):=B(G)\otimes \k$ as a module 
over the group algebra $\k{G_{\ab}}$ by setting 
$\bar{h}\cdot (\bar{g}\otimes 1)= (\bar{h}\cdot \bar{g})\otimes 1$. 

Now suppose $\ch\k=0$.  Then $B_{\k}(G)$ 
has trivial $G_{\ab}$-action if and only if $\bar{h}\cdot 
\bar{g}-\bar{g}$ is a torsion element in $B(G)$, for all 
$\bar{h}$ and $\bar{g}$ as above.  Put another way, $B_{\k}(G)$ 
has trivial $G_{\ab}$-action if and only if, for each $h\in G$ and 
$g\in G'$, there is an integer $m>0$ such that $(h,g)^m\in G''$. 

\subsection{Action of the Torelli group on the Alexander invariant}
\label{subsec:tg bg}

Since both $G'$ and $G''$ are characteristic subgroups of $G$, 
the natural action of $\Aut(G)$ on $G$ induces an action on $B(G)$. 
Explicitly, if $\alpha$ is an automorphism of $G$, and 
$\bar{x}$ is an element in $B(G)=G'/G''$, represented by $x\in G'$, then 
\begin{equation}
\label{eq:ax}
\alpha \cdot \bar{x} = \overline{\alpha(x)}.
\end{equation}
Alternatively, if we identify $B(G)=H_1(G',\Z)$, then $\alpha$ 
acts by the induced homomorphism in homology, $\alpha_*\colon 
H_1(G',\Z) \to H_1(G',\Z)$.  

In general, this action does not respect the $R$-module 
structure on $B(G)$.  Restricting to the Torelli group $\T_G< \Aut(G)$, 
though, remedies this problem.

\begin{prop}
\label{prop:rlin}
The Torelli group $\T_G$ acts $R$-linearly on the Alexander invariant 
$B(G)$.
\end{prop}

\begin{proof}
We need to check that $\alpha\cdot \bar{g}\bar{x} = \bar{g} 
(\alpha \cdot \bar{x})$, for all $\alpha\in \T_G$, and for 
all $g\in G$ and $x\in G'$.  That is, we need to show that 
$\overline{\alpha(g)\alpha(x)\alpha(g)^{-1}} = 
\overline{g \alpha(x) g^{-1}}$, or,
\[
(\alpha(g), y) \equiv (g, y) \,\bmod G'', \quad 
\text{for all $y\in G'$}.
\]

Now, since $\alpha\in \T_G$, we must have $\alpha(g)=zg$, for 
some $z\in G'$.  Using identity \eqref{eq:bilg}, we get 
\begin{align*}
(\alpha(g), y) \cdot (g,y)^{-1} & = {}^z (g,y)\cdot 
(z,y) \cdot (g,y)^{-1}\\
& =(z,(g,y))\cdot ((g,y),(z,y))\cdot (z,y).
\end{align*}
Clearly, this last expression belongs to $G''$, and we are done. 
\end{proof}

As an immediate consequence, we see that the Torelli group $\T_G$ 
acts by $\k{G_{\ab}}$-linear transformations on $B_{\k}(G)$, for 
all fields $\k$.

\subsection{Alexander invariant and associated graded Lie algebras}
\label{subsec:bg}

In \cite{Ms}, W.~Massey  found a simple, yet important relationship 
between the Alexander invariant of a group and the associated 
graded Lie algebra of its maximal metabelian quotient.  

As before, let $B=B(G)$ be the Alexander invariant of $G$, 
viewed as a module over the ring $R=\Z{G_{\ab}}$, and let 
$\gr_I (B)= \bigoplus_{q\ge 0} I^q\cdot B/I^{q+1} \cdot B$ be 
the associated graded module over the ring 
$\gr_I(R)=\bigoplus_{q\ge 0} I^q/I^{q+1}$, where $I$ is 
the augmentation ideal of the group ring $R$. 

Consider also the {\em Chen Lie algebra}\/ of $G$,  
i.e., the associated graded Lie algebra of its maximal 
metabelian quotient, $\gr_{\G} (G/G'')$.

\begin{prop}[\cite{Ms}]
For each $q\ge 0$, there is a natural isomorphism
\begin{equation}
\label{eq:mass}
\gr_I^q (B) \cong \gr_{\G}^{q+2} (G/G'')  \, .
\end{equation}
\end{prop}

The next proposition provides another link between 
the Alexander invariant and the associated graded 
Lie algebra of a group.

\begin{prop}
\label{prop:alex gr}
Let $G$ be a finitely generated group, with Alexander invariant 
$B(G)$.  Assume $\ch (\k)=0$.  If $B_{\k}(G)$ has trivial 
$G_{\ab}$-action, then $\gr^3_{\Gamma}(G) \otimes \k=0$.  
\end{prop}

\begin{proof}
For any group $G$ with lower series terms $\Gamma^s=\Gamma^s(G)$, 
we have $(G,G')=(\Gamma^1,\Gamma^2)= \Gamma^3$, 
and $G''=(G',G')=(\Gamma^2,\Gamma^2)\subseteq \Gamma^4$.  
Thus, if $B_\k(G)$ has trivial $G_{\ab}$-action then, for each $x\in \G^3$ 
there is an integer $m>0$ such that $x^m \in \Gamma^4$.
In other words, $\gr^3_{\Gamma}(G)$ must be a torsion 
$\Z$-module. Since $\k$ has characteristic $0$, we conclude 
that $\gr^3_{\Gamma}(G)\otimes \k=0$.
\end{proof}

\subsection{Holonomy Lie algebra}
\label{subsec:holo lie}
Next, we review a notion that goes back to the work of 
K.-T. Chen \cite{Ch}.  For more details on this material,  
we refer to \cite{PS-imrn}, and references therein.

Let $G$ be a finitely generated group. Write $H=H_1(G,\Z)$ 
and $H_{\C}=H_1(G,\C)=H\otimes \C$.  Let $\Lie_{\bullet}(H_{\C})$ 
be the free Lie algebra generated by the vector space $H_{\C}$, 
with grading given by bracket length. Use the Lie bracket to identify 
$\Lie_{2}(H_{\C})$ with $H_{\C}\wedge H_{\C}$. 

Denote by $\partial_G\colon H_2(G,\C)\to H_{\C}\wedge H_{\C}$
the comultiplication map, that is, the linear map whose dual is the 
cup-product map, $\cup_G\colon H^1(G,\C)\wedge H^1(G,\C)\to H^2(G,\C)$.
The {\em holonomy Lie algebra}\/ of $G$, denoted $\h_{\bullet} (G)$, 
is the quotient of $\Lie_{\bullet}(H_{\C})$ by the (quadratic) Lie ideal 
generated by the image of $\partial_G$,
with grading inherited from the free Lie algebra. 
Since the dual of $\partial_G$ is $\cup_G$, 
it follows that the graded Lie algebra $\h_{\bullet} (G)$ depends 
only on the co-restriction of $\cup_G$ to its image.

Recall that the associated graded Lie algebra $\gr_{\G}(G)$ is 
generated in degree $1$. Thus, there is a natural 
epimorphism $\Lie_{\bullet}(H_{\C}) \surj \gr_{\G}^{\bullet} (G)\otimes \C$, 
induced by the identification $H_{\C} = G_{\ab}\otimes \C$.  
Using the relation between $\cup_G$ and the group 
commutator, due to D.~Sullivan, it is readily seen that 
this morphism factors through the holonomy Lie algebra, 
cf.~\cite{PS-imrn, DP}.  Thus, we have an epimorphism 
\begin{equation}
\label{eq:hepi}
\xymatrix{\h_{\bullet} (G) \ar@{->>}[r] & \gr_{\G}^{\bullet} (G)\otimes \C}.
\end{equation}

For an arbitrary Lie algebra $\g$, denote by $\g'=[\g, \g]$ the derived 
Lie algebra, and by $\g''=[\g', \g']$ the second derived Lie algebra.  
When $\g=\g_{\bullet}$ is graded, and generated as a Lie algebra by 
$\g_1$, we have $\g'=\g_{\ge 2}$.  

Let us return now to our group $G$. Composing the canonical projection 
$\gr_{\G}^{\bullet} (G)\otimes \C\surj \gr_{\G}^{\bullet} (G/G'')\otimes \C$ 
with the epimorphism \eqref{eq:hepi}, we obtain a morphism which factors 
through the metabelian quotient $\h/\h''$ to yield an epimorphism 
\begin{equation}
\label{eq:bepi}
\xymatrix{\h_{\bullet} (G)/ \h''_{\bullet} (G)\ar@{->>}[r] 
& \gr_{\G}^{\bullet} (G/G'')\otimes \C}.
\end{equation}

\subsection{Infinitesimal Alexander invariant}
\label{ss52}
As before, let $\h=\h_{\bullet}(G)$ be the holonomy Lie algebra 
of $G$.  The universal enveloping algebra of the abelian Lie algebra 
$\h/ \h'=H_{\C}$ may be identified with the symmetric algebra 
$S=\Sym (H_{\C})$.  In turn, $S$ may be viewed as a polynomial 
algebra, endowed with the usual grading. 

Following \cite{PS-imrn}, let us consider the 
{\em infinitesimal Alexander invariant}\/ of $G$, defined as 
\begin{equation}
\label{eq:infb}
\B_{\bullet}(G)= \h_{\bullet}'(G)/\h_{\bullet}''(G). 
\end{equation}
The exact sequence of graded Lie algebras
\begin{equation}
\label{eq:defb}
\xymatrix{0 \ar[r]& \h'/\h'' \ar[r] 
& \h/\h'' \ar[r]& \h/\h' \ar[r]&  0}
\end{equation}
yields a positively graded $S$-module structure on 
$\B_{\bullet}(G)$, with grading coming from the one on 
$\h_{\bullet}(G)$, and with $S$-action defined by the 
adjoint action. Explicitly, a monomial $x_1\cdots x_l \in S$
in variables $x_i \in \h/\h'$ represented by $a_i\in \h$ acts on 
an element $\bar{b} \in \h'/\h''$ represented by $b\in \h'$ as 
\begin{equation}
\label{eq:s-action}
(x_1\cdots x_l) \cdot \bar{b}= 
\overline{\ad_{a_1}\circ \dots \circ \ad_{a_l} (b)}.
\end{equation}

In \cite[Theorem 6.2]{PS-imrn}, the following presentation 
of $\B(G)$ by free $S$-modules was given:
\begin{equation} 
\label{eq:lin alex pres}
\xymatrix{
\left(S \otimes_{\C} \bigwedge^3 H_{\C}\right) \oplus 
(S\otimes_{\C} H_2(G,\C)) \ar^(.65){\delta _3 + 
\id \otimes \partial _G}[rr] &&
S \otimes_{\C}  \bigwedge^2 H_{\C} \ar@{->>}[r] & \B(G)
}.
\end{equation}
Here, the group $\bigwedge^3 H_{\C}$ is in degree $3$, 
the groups $H_2(G,\C)$ and $\bigwedge^2 H_{\C}$ are in 
degree $2$, and the $S$-linear map $\delta_3$ is given by 
$ \delta _3 (x \wedge y  \wedge z)
= x\otimes y   \wedge z -y \otimes x  \wedge z +z 
\otimes x \wedge y$.

\subsection{$1$-Formality}
\label{subsec:1formal}
We conclude this section by recalling a basic notion from rational homotopy 
theory, introduced by D.~Sullivan in \cite{Su}.  A finitely generated group 
$G$ is said to be {\em $1$-formal}\/ if the Malcev Lie algebra  
$\m(G)$---constructed by D.~Quillen \cite{Q}---is the completion of a 
quadratic Lie algebra.  Relevant to us is the following result, which  
relates the $B$- and $\B$-modules of a group in this formal setting.

\begin{theorem}[\cite{DPS-duke}]
\label{thm:dps-bmod}
Let $G$ be a $1$-formal group.  There is  then a filtered isomorphism 
\begin{equation}
\label{eq:bbcompl}
\widehat{B_{\C}(G)}\cong \widehat{\B(G)} 
\end{equation}
between the $I$-adic completion of the Alexander invariant of $G$ 
and the degree completion of the infinitesimal Alexander invariant of $G$.
\end{theorem}

Here, $\widehat{B_{\C}(G)}$ is viewed as a module over the $I$-adic completion 
of the group ring $\C{G_{\ab}}$, where $I$ is the augmentation ideal, 
while $\widehat{\B(G)}$ is viewed as a module over the $m$-adic completion 
of the polynomial ring $S$, where $m$ is the maximal ideal at $0$. 
The isomorphism \eqref{eq:bbcompl} covers the canonical identification 
of $\widehat{\C{G_{\ab}}}$ with the ring of formal power series $\widehat{S}$.

\begin{example}
\label{ex:free formal}
The free group $F_n$ is $1$-formal. Indeed, the holonomy Lie algebra 
$\h(F_n)$ is the (complex) free Lie algebra, $\L_n\otimes \C$, 
while the Malcev Lie algebra $\m(F_n)$ is just the degree 
completion of $\h(F_n)$. In this case, $\widehat{S}=\C[[ x_1,\dots, x_n]]$ 
and $\widehat{\B (F_n)}$ is the cokernel of the Koszul differential 
$\delta_3 \colon 
\widehat{S}\otimes_{\C}  \bigwedge^3 \C^n \to 
\widehat{S}\otimes_{\C}  \bigwedge^2 \C^n$. 
\end{example}

Both the pure braid group $P_n$ and the  McCool 
group $\PS_n$ (which we shall encounter in \S\ref{subsec:mccool}) 
are $1$-formal. 
For more details and references on the topic of $1$-formality, 
we refer to \cite{PS-imrn} and the recent survey \cite{PS-bmssmr}.

\section{Characteristic varieties and homology of abelian covers}
\label{sect:hac}

We devote this section to an analysis of various homological 
finiteness properties of Alexander type invariants. We refer 
the reader to the book of Eisenbud \cite{E} for standard tools 
from commutative algebra.

\subsection{Support and characteristic varieties}
\label{subsec:scv}

Let $X$ be a connected CW-complex, with finite $1$-skeleton 
$X^{(1)}$, basepoint $x_0$, and (finitely generated) fundamental 
group $G=\pi_1(X,x_0)$.   Let $\nu\colon G\surj A$ be an 
epimorphism onto an abelian group $A$,
and denote by $X^{\nu}$ the corresponding Galois cover of $X$. 

Fix a coefficient field $\k$.  
The group algebra $R=\k{A}$ is a commutative ring, which is 
finitely generated as a $\k$-algebra. 
The action of $A$ on $X^{\nu}$ by deck-transformations induces 
an $R$-module structure on the homology groups $H_*(X^{\nu}, \k)$.

Assume now that the $q$-skeleton of $X$ is finite, for some fixed $q\ge 1$.  
Then, for each $i\le q$, the homology group $M=H_i(X^{\nu}, \k)$ is 
a finitely generated module over the Noetherian ring $R$. 
Hence, the support of this module,  
\begin{equation}
\label{eq:supp}
\supp (M):=V(\ann (M)), 
\end{equation}
is a Zariski closed subset in the maximal 
spectrum $\specm (R)$.

From now on, we will also assume, without any essential 
loss of generality, that the field $\k$ is algebraically closed. 
In this case, the ring $R=\k{A}$ is the coordinate ring of the affine 
algebraic group $\TT_{\k} (A)=\Hom (A, \k^{\times})$, the 
character group of $A$. Furthermore, the augmentation 
ideal $I\triangleleft \k{A}$ is the maximal ideal vanishing 
at the unit $1\in \TT_{\k}(A)$.  Clearly, $M=H_0(X^{\nu}, \k)=\k$ 
is an $R$-module with trivial $A$-action (that is, $I\cdot M=0$), 
and $\ann(M)=I$.  Consequently, $\supp(M)= \{1\}$. 

The character group 
$\TT_{\k} (G)=\Hom (G, \k^{\times})= \Hom (G_{\ab}, \k^{\times})$ 
parametrizes all rank one local systems on $X$.  For each 
character $\rho\colon G\to \k^{\times}$, denote by $\k_{\rho}$ 
the $1$-dimensional $\k$-vector space, viewed as a 
$\k{G}$-module via the action $g\cdot a=\rho(g) a$. 
For each $0\le i\le q$ and $d>0$, define 
\begin{equation}
\label{eq:defcv}
\VV^i_d (X, \k)= \{ \rho\in \Hom (G, \k^{\times}) \mid 
\dim_{\k} H_i(X, \k_{\rho})\ge d \}. 
\end{equation}
These sets,  called the {\em characteristic varieties}\/ of $X$, 
are Zariski closed subsets of the algebraic group $\TT_{\k}(G)$, and depend 
only on the homotopy type of $X$.  In particular, $\VV^0_1 (X, \k)=\{ 1\}$, 
and $1\in \VV^i_1(X,\k)$ if and only if $H_i(X,\k)\ne 0$. 

Given a finitely generated group $G$, pick a classifying 
space $X=K(G,1)$ with finitely many $q$-cells, for some 
$q\ge 1$, and set $\VV^i_d (G, \k):=\VV^i_d(X,\k)$.  In the 
default situation when $\k=\C$, we abbreviate $\TT_{\C}(G)$ 
by $\TT(G)$ and $\VV^1_1 (G, \C)$ by $\VV(G)$.

\begin{example}
\label{ex:fn}
Let $G=F_n$ be a free group of rank $n\ge 1$, with generators 
$x_1,\dots, x_n$.   In this case, $X=K(G,1)$ is a bouquet of $n$ 
circles, and the cellular chain complex of the universal cover 
$\widetilde{X}$ is concentrated in degrees $1$ and $0$,
with $\Z F_n$-linear differential, $(\Z F_n)^n \to \Z F_n$, 
given by the vector $(x_1-1, \dots, x_n-1)$. 

Now identify the character group $\TT_\k(F_n)$ 
with the algebraic torus $(\k^{\times})^n$. Under 
this identification, a character $\rho\colon F_n\to \k^{\times}$ 
corresponds to the point in $(\k^{\times})^n$ whose 
$i$-th coordinate is $\rho(x_i)$. 
Thus, the twisted homology $H_{\bullet}(X, \k_{\rho})$ 
is the homology of the specialized chain complex,
$\k^n \to \k$, with $\k$-linear differential given by 
the vector $(\rho(x_1)-1, \dots, \rho(x_n)-1)$.

It follows from definition \eqref{eq:defcv} that  
$\VV^1_d(F_n, \k)=(\k^{\times})^n$ for $d\le n-1$, 
while $\VV^1_n(F_n,\k)=\{1\}$ and 
$\VV^i_d(F_n,\k)=\emptyset$ for $i>1$ or $d>n$. 
\end{example}

Characteristic varieties enjoy several functoriality properties.  
For instance, suppose $\nu\colon G\surj K$ is an epimorphism 
between finitely generated groups. Then, the induced algebraic morphism 
between character groups, $\nu^*\colon \TT_\k(K)\inj \TT_\k(G)$,  
restricts to an embedding $\VV^1_d(K, \k) \inj \VV^1_d(G, \k)$, 
for each $d\ge 1$.

\subsection{Alexander polynomial}
\label{subsec:alex poly}

Let $G$ be a finitely generated group, and let $H$ be its 
maximal torsion-free abelian quotient.  Identify the group 
ring $R=\Z{H}$ with the ring of Laurent polynomials in 
$n$ variables, where $n=b_1(G)$.  

The {\em Alexander module}\/ of $G$ is the $\Z{H}$-module 
defined as $A(G)=\Z{H}\otimes_{\Z{G}} I_{G}$, where $I_{G}$ 
is the augmentation ideal of $\Z{G}$. Since $\Z{H}$ is a 
Noetherian ring, the finite generation of $G$ implies the 
finite presentability of the $\Z{H}$-module $A(G)$. The
first elementary ideal of $A(G)$ is the ideal of $\Z{H}$ 
generated by the codimension $1$ minors of a finite 
$\Z{H}$-presentation matrix for $A(G)$. By standard 
commutative algebra (see for instance \cite{E}), this ideal 
is independent of the chosen presentation for $A(G)$. 
 
The {\em Alexander polynomial}, $\Delta_G$, is 
the greatest common divisor of all elements in the 
first elementary ideal of $A(G)$.  It is readily seen that 
the polynomial $\Delta_G\in \Z{H}$ depends only on the 
group $G$, modulo units in $\Z{H}$.  

As shown in \cite{DPS-imrn}, the Alexander polynomial 
of $G$  is determined by the first characteristic variety 
of the group.  More precisely, let 
$\cv (G)$ be the union of all codimension-one 
irreducible components of $\VV(G) \cap \TT(G)^0$, 
where $\TT(G)^0\cong \TT(H)$ is the identity 
component of the character group.    

\begin{theorem}[\cite{DPS-imrn}]  
\label{thm:deltav1}
For a finitely generated group $G$, the following hold:
\begin{enumerate}
\item \label{dc1}
$\Delta_G=0$ if and only if 
$\TT(G)^0 \subseteq \VV(G)$.
In this case, $\cv (G)=\emptyset$.
\item \label{dc2}
If $b_1(G)\ge 1$ and $\Delta_G\ne 0$, then
\begin{equation*}
\cv (G) =\begin{cases}
V(\Delta_G) & \text{if $b_1(G)>1$}\\
V(\Delta_G)\coprod \{ 1\}  & \text{if $b_1(G)=1$}.
\end{cases}
\end{equation*}
\item \label{dc3} If $b_1(G)\ge 2$, then 
$\cv (G)=\emptyset$ if and only if $\Delta_G$ is a 
constant, up to units, i.e., $\Delta_G=c\cdot u$, 
for some $c\in \Z \le \Z{H}$ and $u\in (\Z{H})^{\times}$. 
\end{enumerate}
\end{theorem}

\begin{corollary}
\label{cor:delta const} 
Let $G$ be a finitely generated group, with $b_1(G)\ge 2$.  
If $\VV(G)$ is finite, then $\Delta_G$ equals, up to units, 
a non-zero constant. 
\end{corollary}

The next example illustrates the necessity of the condition 
on the first Betti number.

\begin{example}
\label{ex:delta knot}
Let $k$ be a tame knot in $S^3$, and let $G$ be the fundamental 
group of the knot complement.  By Theorem \ref{thm:deltav1}\eqref{dc2}, 
the characteristic variety $\VV(G)\subset \C^{\times}$ consists 
of $1$, together with the roots of the Alexander polynomial 
of the knot, $\Delta_G \in \Z[t^{\pm 1}]$.  In particular, $\VV(G)$ 
is always finite, yet typically, $\Delta_G$ is not constant.
\end{example}

\subsection{Dwyer-Fried test}
\label{subsec:df}

Returning to the setup from \S\ref{subsec:scv}, let $X$ be a 
connected CW-complex, with fundamental group $G=\pi_1(X,x_0)$,
and $\nu\colon G \surj A$ 
be an homomorphism onto an abelian group $A$.  
In \cite{DF}, Dwyer and Fried proved a remarkable 
result: assuming $X$ is finite and $A$ has no torsion, 
they characterized the finite-dimensionality of the 
$\C$-vector space $H_*(X^{\nu}, \C)$ in terms of the 
intersection of the support of the homology of the 
universal free abelian cover with the character group of $A$. 

Using the approach from \cite{PS-plms}, this result may be 
extended---with a different proof---to a more general situation. 
Fix a coefficient field $\k$ with $\k=\bar{\k}$.

\begin{theorem}
\label{thm:dfext}
Let $X$ be a connected CW-complex with finite $q$-skeleton, 
for some $q\ge 1$, and let $G$ be its fundamental group. 
Let $\nu\colon G \surj A$ be an 
epimorphism onto an abelian group $A$, and denote by 
$\nu^*\colon \TT_\k(A) \inj \TT_\k(G)$ the induced 
homomorphism on character groups.  Then
\begin{equation}
\label{eq:df}
\sum_{i=1}^q \dim_{\k}  H_{i}(X^{\nu}, \k)<\infty \same 
\im (\nu^*) \cap \Big( \bigcup_{i=1}^q \VV^i_1 (X, \k) \Big)\
\text{is finite}.
\end{equation}
\end{theorem} 

\begin{proof}
By \cite[Theorem 3.6]{PS-plms}, we have
\begin{equation}
\label{eq:ps}
\bigcup_{i=0}^q \supp (H_i(X^{\nu}, \k)) = 
\im (\nu^*) \cap \left( \bigcup_{i=0}^q \VV^i_1 (X, \k) \right) .
\end{equation}
Now, if $M$ is a finitely generated module over a finitely 
generated $\k$-algebra $R$, then $\dim_{\k}M<\infty$ if and 
only if $\supp(M)$ is a finite subset of $\specm (R)$. 
In view of the discussion from \S\ref{subsec:scv}, this 
finishes the proof.
\end{proof}

The Alexander invariant has a well-known topological 
interpretation.  Let $X=K(G, 1)$ be a classifying space with finite 
$1$-skeleton, and let $X^{\ab}$ be the universal abelian cover.  
Then $B_{\k}(G)= H_1(G', \k)$ is isomorphic to the module 
$M=H_1(X^{\ab}, \k)$ discussed in \S\ref{subsec:scv}. 

\begin{corollary}
\label{cor:btest}
Let $G$ be  a finitely generated group, and $\k$ 
an algebraically closed field. 
\begin{enumerate}
\item \label{bt1}
The support $\supp (B_{\k}(G))$ is equal (away from $1$) 
to the variety $\VV^1_1(G, \k)$.
\item \label{bt2}
The $\k$-vector space $B_{\k}(G)$ is finite-dimensional if and 
only if $\VV^1_1(G, \k)$ is finite.
\end{enumerate}
\end{corollary}

\begin{proof}
Part \eqref{bt1} is proved in  \cite[Theorem 3.6]{PS-plms}, 
while part \eqref{bt2} follows from Theorem \ref{thm:dfext}.
\end{proof}

For instance, take $G$ to be the free group $F_n$, $n\ge 2$.  
From Example \ref{ex:fn}, we know that 
$\VV^1_1(F_n,\k)= (\k^{\times})^n$.  
Since $\k=\bar{\k}$, this set is infinite, 
and thus $\dim_{\k} B_{\k}(F_n) = \infty$. 

\subsection{Nilpotency and finiteness properties}
\label{ss44}

As before, let $A$ be a finitely generated abelian 
group, and let $\k$ be an algebraically closed field. 
Set $R=\k{A}$, and let $I\subseteq R$ be the 
augmentation ideal. 
A finitely generated $R$-module $M$ is said to 
be {\em nilpotent}\/ if $I^m\cdot M=0$,
for some $m>0$. Clearly, every module with trivial $A$-action is nilpotent. 

\begin{lemma}
\label{lem:nilpv}
The module $M$ is nilpotent if and only if $\supp (M)\subseteq \{1\}$.
\end{lemma}

\begin{proof}
Plainly, $\supp (M)\subseteq \{1\}$ if and only if $V(\ann(M)) 
\subseteq V(I)$.  By the Hilbert Nullstellensatz, this last 
inclusion is equivalent to $I^m \subseteq \ann(M)$, for some $m>0$. 
\end{proof}

The typical example we have in mind is the (complexified) 
Alexander invariant of a finitely generated group $G$.  
In this situation, $R=\C{G_{\ab}}$ and $M=H_1(G', \C)=B_{\C}(G)$.

For instance, let $G$ be a finitely generated nilpotent group.  Then 
the module $B_{\C}(G)$ is nilpotent.  Indeed, according 
to \cite{MP}, we have in this case $\VV(G)\subseteq \{1\}$, and 
the claim follows from Lemma \ref{lem:nilpv} 
and Corollary \ref{cor:btest}\eqref{bt1}. 
However, this module need not have trivial $G_{\ab}$-action.

\begin{example}
\label{ex:nilptriv}
Let $G=F_n/\G^4(F_n)$ be the free, $3$-step nilpotent 
group of rank $n\ge 2$.  Then $B_{\C}(G)$ does {\em not}\/ 
have trivial $G_{\ab}$-action.
For, otherwise, Proposition \ref{prop:alex gr} would imply 
$\gr_{\G}^3 (G) \otimes \C=0$.  On the other hand, 
$\gr_{\G}^3 (G)\otimes \C=\gr_{\G}^3 (F_n)\otimes \C$, 
and the latter group is non-zero by 
Theorem \ref{thm:free lie}\eqref{fl4}.
\end{example}

Returning to the general situation, let $M$ be an $R$-module as above. 
The $I$-adic completion, $\widehat{M}=\varprojlim_{r} M/I^r\cdot M$, 
is a module over the ring $\widehat{R}=\varprojlim_{r} R/I^r$.  
Denote by $I^{\infty}M = \bigcap_{r\ge 0} I^r \cdot M$ the kernel 
of the completion map, $M\to \widehat{M}$.  The proof of the next 
lemma is straightforward.

\begin{lemma}
\label{lem:finm}
For a finitely generated $\k{A}$-module $M$, the following hold.
\begin{enumerate}
\item \label{fin1}
$M$ nilpotent $\Longrightarrow \dim_{\k} M<\infty 
\Longrightarrow \dim_{\k} \widehat{M}<\infty$.
\item \label{fin2} 
Assume $\dim_{\k} M<\infty $.  Then $M$ is nilpotent 
if and only if $I^{\infty}M=0$. 
\item \label{fin3} 
Assume $\dim_{\k} \widehat{M}<\infty$.  Then 
$\dim_{\k} M<\infty$ if and only if $\dim_{\k} I^{\infty}M<\infty$.
\end{enumerate}
\end{lemma} 

As we shall see in Examples \ref{ex:knots} and \ref{ex:orbi} below, 
neither implication in Part \eqref{fin1} can be reversed, 
even for (complex) Alexander invariants.

\begin{example}
\label{ex:knots}
Let $G$ be a knot group.  As we saw in Example \ref{ex:delta knot}, 
the characteristic variety $\VV(G)$ is finite. Thus, by 
Corollary \ref{cor:btest}\eqref{bt2}, the $\C\Z$-module 
$M=B_{\C}(G)$ satisfies $\dim_{\C} M< \infty$; in fact,
as is well-known, $\dim_{\C} M = \deg \Delta_G$.

Now assume the Alexander polynomial of the knot is 
not equal to $1$.  Then, by Corollary \ref{cor:btest}\eqref{bt1}, 
we must have $\supp(M) \not\subseteq \{1\}$.  Hence, by 
Lemma \ref{lem:nilpv}, the module $M$ is not nilpotent.  
This shows that the first implication from 
Lemma \ref{lem:finm}\eqref{fin1} cannot be reversed.
\end{example}

We finish this section with a class of groups for which 
the complexified Alexander invariant is infinite-dimensional. 
A group $G$ is said to be {\em very large}\/ if it 
surjects onto a non-abelian free group.

\begin{prop}
\label{prop:cork}
Let $G$ be a finitely generated group.  
If $G$ is very large, then $\dim_{\C} B_{\C}(G) =\infty$.
\end{prop}

\begin{proof}
Let $\nu\colon G\surj F_n$ be an epimorphism onto a free group of 
rank $n\ge 2$.   By the discussion at the end of \S\ref{subsec:scv}, 
the morphism $\nu^*\colon \TT(F_n)\to \TT(G)$ embeds 
$\VV(F_n)= (\C^{\times})^n$ into $\VV(G)$.  Hence, the set  
$\VV(G)$ is infinite. The desired conclusion follows from 
Corollary \ref{cor:btest}\eqref{bt2}.
\end{proof}

We shall see in \S\ref{subsec:mccool} some concrete examples where 
this Proposition applies.

\section{Resonance varieties and the dimension of the Alexander invariant}
\label{sect:resalex}

Under favorable circumstances, we show that there is an (attainable) 
upper bound for the size of the Alexander invariant, a bound that can 
be described using only untwisted cohomological information 
in low degrees. 

\subsection{Resonance varieties}
\label{subsec:cjl}

As before, let $X$ be a connected CW-complex with finite $q$-skeleton, 
and let $\k$ be a field, with $\ch\k\ne 2$.  Given a cohomology class  
$z\in H^1(X,\k)$, left-multiplication by $z$ turns the cohomology 
ring $H^*(X,\k)$ into a cochain complex. The {\em resonance varieties}\/ 
of $X$ are the jump loci for the cohomology of this cochain complex.  
More precisely, for each $0\le i\le q$ and $d>0$, define 
\begin{equation}
\label{eq:defres}
\RR^i_d (X, \k)= \{ z\in H^1 (X, \k) \mid 
\dim_{\k} H^i(H^*(X,\k), \cdot z)\ge d \}. 
\end{equation}
These sets are homogeneous subvarieties of the affine space 
$H^1 (X, \k)$, and depend only on the homotopy type of $X$.  
In particular, $\RR^0_1 (X, \k)=\{ 0\}$, and $0\in \RR^i_1(X,\k)$ 
if and only if $H_i(X,\k)\ne 0$. 

Now suppose $G$ is a finitely generated group.  Pick a classifying 
space $X=K(G,1)$ with finitely many $1$-cells, and set 
$\RR^i_d (G, \k):=\RR^i_d(X,\k)$.  We will only be interested 
here in the case when $\k=\C$ and $i=d=1$. To simplify notation, 
we shall write $\RR(G)=\RR^1_1 (G, \C)$.  Concretely, 
\begin{equation}
\label{eq:rv}
\RR (G)=\{z \in H^1(G,\C)  \mid  \text{$\exists u\in H^1(G,\C), 
u \notin \C \cdot z$,  and $z\cup u=0$} \} \, , 
\end{equation}
where $\cup\colon H^1(G,\C)\wedge  H^1(G,\C)\to H^2(G,\C)$ 
is the cup-product map in low degrees. Clearly, the variety 
$\RR (G)$ depends only on the co-restriction of $\cup$ to its image.  
In fact, the resonance variety $\RR(G)$ coincides, 
away from the origin $0\in H^1(G,\C)$, with the 
support variety $V(\ann (\B(G)))$. 

The jump loci $\VV(G)$ and $\RR(G)$ are related, as follows. 
Let $1$ be the identity of the character group $\TT(G)=H^1(G,\C^{\times})$, 
and let $\TC_1(\VV(G))$ be the tangent cone at $1$ to the characteristic 
variety, viewed as a subset of the tangent space $T_1(\TT(G))=H^1(G,\C)$. 

\begin{theorem}[\cite{Li}, \cite{DPS-duke}]
\label{thm:tcone}
Let $G$ be a finitely generated group. 
\begin{enumerate}
\item \label{t1} 
$\TC_1(\VV(G))\subseteq \RR(G)$.

\item \label{t2} 
If $G$ is $1$-formal, then $\TC_1(\VV(G)) = \RR(G)$.
\end{enumerate}
\end{theorem}

Part \eqref{t1} of this theorem was proved by Libgober \cite{Li} 
for finitely presented groups, while the extension to finitely generated 
groups was given in \cite[Lemma 5.5]{DP}  (see also the discussions 
from \cite[\S 2.6]{DPS-imrn} and \cite[\S 2.2]{PS-plms}).  Finally, 
Part \eqref{t2} was proved in \cite[Theorem A]{DPS-duke}.  
As an immediate corollary, we have the following result.

\begin{corollary}
\label{cor:zerores}
Let $G$ be a finitely generated group. 
\begin{enumerate}
\item \label{rz1} 
If $\RR(G)\subseteq \{0\}$, then either $1\notin \VV(G)$,  
or $1$ is an isolated point in $\VV(G)$.  

\item \label{rz2}
If $\VV(G)\subseteq \{1\}$ and $G$ is $1$-formal, 
then $\RR(G)\subseteq \{0\}$. 
\end{enumerate}
\end{corollary}

The vanishing of the resonance variety also controls 
certain finiteness properties related to the Alexander invariants.

\begin{theorem}[\cite{DP}]
\label{prop:res bg}
Let $G$ be a finitely generated group. 
\begin{enumerate}
\item \label{res1}
If $\RR (G)\subseteq  \{ 0\}$ then $\dim_{\C} \widehat{B_{\C}(G)}< \infty$.
When $G$ is $1$-formal, the converse holds as well.

\item \label{res2}
$\RR (G)\subseteq  \{ 0\}$ if and only if $\dim_{\C} \B (G)< \infty$.
\end{enumerate}
\end{theorem}

Here, recall that $\widehat{B_{\C}(G)}$ denotes the $I$-adic completion 
of the $R$-module $B_{\C}(G)=H_1(G',\C)$, where $R$ is the 
group algebra $\C{G_{\ab}}$ and 
$I$ is the augmentation ideal, while $\B(G)$ denotes the 
$S$-module defined by \eqref{eq:infb}, and $S$ is the 
polynomial ring $\Sym(G_{\ab}\otimes \C)$.

\subsection{An upper bound for the Alexander invariant}
\label{ss5end}

We are now ready to state and prove the main result in this section.

\begin{theorem}
\label{thm:cv1}
Let $G$ be a finitely generated group. 
\begin{enumerate}
\item \label{v1}
If $\VV (G) \subseteq \{1\}$, then 
$\dim_{\C} B_{\C}(G) \le \dim_{\C} \B(G)$.

\item \label{v2}
If $\VV (G) \subseteq \{1\}$ and $\RR (G) \subseteq \{0\}$, then 
$\dim_{\C} B_{\C}(G) \le \dim_{\C} \B(G) <\infty$.

\item \label{v3}
If $\VV (G) \subseteq \{1\}$ and $G$ is $1$-formal, then 
$\dim_{\C} B_{\C}(G) = \dim_{\C} \B(G) <\infty$.
\end{enumerate}
\end{theorem}

\begin{proof}
{\em Part~\eqref{v1}}.
Assuming that $\VV (G) \subseteq \{1\}$,  
Lemma \ref{lem:nilpv} and Corollary \ref{cor:btest} imply 
that the module $B_{\C}(G)$ is nilpotent. Therefore, 
\begin{equation}
\label{eq:hilbbb}
\dim_{\C} B_{\C}(G) = \dim_{\C} \widehat{B_{\C}(G)}= 
\sum_{q\ge 0} \dim_{\C} \gr_I^q (B_{\C}(G)).
\end{equation}
Furthermore, formulas \eqref{eq:mass} and \eqref{eq:bepi} 
give  
\begin{equation}
\label{eq:hilb again}
\sum_{q\ge 0} \dim_{\C} \gr_I^q (B_{\C}(G)) \le \dim_{\C} \B(G).
\end{equation}
Thus, $\dim_{\C} B_{\C}(G) \le \dim_{\C} \B(G)$.
\medskip

{\em Part~\eqref{v2}}.
Assuming that $\RR (G) \subseteq \{0\}$, 
Theorem \ref{prop:res bg}\eqref{res2} insures that  
$\dim_{\C}\B (G) < \infty$.  By the above, we are done.
\medskip

{\em Part~\eqref{v3}}.
Assuming that $G$ is $1$-formal and $\VV (G) \subseteq \{1\}$, 
Corollary \ref{cor:zerores} gives $\RR (G) \subseteq \{0\}$.   
Hence, as we just saw, $\dim_{\C} \B (G) < \infty$. 

On the other hand, according to \eqref{eq:bbcompl}, the 
$1$-formality assumption on our group $G$ also implies that 
$\widehat{B_{\C}(G)}$ is isomorphic to $\widehat{\B (G)}$, 
the degree completion of the infinitesimal Alexander invariant. 
But, since  $\B (G)$ has finite dimension, 
$\widehat{\B (G)} \cong \B (G)$.  Using 
once again \eqref{eq:hilbbb}, we find that 
\begin{equation}
\label{eq:dim b}
\dim_{\C} B_{\C}(G) = \dim_{\C} \widehat{B_{\C}(G)} = 
\dim_{\C} \widehat{\B (G)} = \dim_{\C} \B (G) < \infty,
\end{equation}
and this finishes the proof.
\end{proof}

If we replace the condition $\VV (G)\subseteq \{ 1\}$ by 
$\RR (G) \subseteq \{0\}$ in Theorem \ref{thm:cv1}, 
the inequality $\dim_{\C} B_{\C}(G) \le \dim_{\C} \B(G)$ fails, 
even for $1$-formal groups. 

\begin{example}
\label{ex:orbi}
Fix an integer $m\ge 2$, and consider the group 
$G= \langle x,y \mid y^m=1 \rangle$.  Then $G$ 
is $1$-formal, and has abelianization $G_{\ab}= \Z \oplus \Z_m$. 
Since $H^1(G,\C)=\C$ and $H^2(G,\C)=0$, we infer that $\RR (G)=\{ 0\}$. 
Hence, by Theorem \ref{prop:res bg}\eqref{res2}, we have 
$\dim_{\C} \B (G)<\infty$.

The character group $\TT (G)$ consists of $m$ disjoint copies of $\C^{\times}$. 
Starting from the defining presentation of $G$, a Fox calculus computation 
shows that $\VV (G)$ consists of the identity $1$, together with the union 
of all connected components of $\TT (G)$ not containing $1$. 
Hence, by Corollary \ref{cor:btest}\eqref{bt2}, we have 
$\dim_{\C} B_{\C}(G) =\infty$.

On the other hand, Theorem \ref{prop:res bg}\eqref{res1} 
implies that $\dim_{\C} \widehat{B_{\C}(G)}<\infty$.  
Thus, this example also shows that the second implication 
from Lemma \ref{lem:finm}\eqref{fin1} cannot be reversed, 
even in the case when the module $M$ is the complexified 
Alexander invariant of a $1$-formal group.
\end{example}

\section{Automorphism groups of free groups}
\label{sect:autfn}

We now specialize to the case of Torelli groups associated to 
a finitely generated free group $F_n$. Exploiting the structure of 
various braid-like subgroups of $\Aut(F_n)$, we extract some 
information on the associated graded Lie algebras and the 
Alexander invariants of those Torelli groups.

\subsection{The free group and its automorphisms}
\label{subsec:fn}

Let $F_n$ be the free group on generators $x_1,\dots, x_n$, 
and let $\Z^n$ be its abelianization.   Identify the automorphism 
group $\Aut(\Z^n)$ with the general linear group $\GL_n(\Z)$.   
As is well-known, the map $\Aut(F_n)\to \GL_n (\Z)$ 
which sends an automorphism to the induced map on the
abelianization is surjective.   Thus, we may identify the 
symmetry group $\A(F_n)$ with $\GL_n(\Z)$.   

The Torelli group $\T_{F_n}=\ker\, (\Aut(F_n) \surj \GL_n(\Z))$ is 
classically denoted by $\IA_n$.  This group is the first term in 
the Johnson filtration $J^s_n=F^s(\Aut(F_n))$, which, in this context, 
goes back to Andreadakis \cite{An}.  

Let $\Inn(F_n)$ be the group of inner automorphisms of $F_n$. 
If $n\ge 2$, the free group $F_n$ is centerless, and so 
$\Inn(F_n)\cong F_n$. As usual, we denote by 
$\Out(F_n)=\Aut(F_n)/\Inn(F_n)$ 
the corresponding outer automorphism group.

As we know, $\Inn(F_n)$ is also a normal 
subgroup of $\IA_n$.  The quotient group, $\OA_n=\IA_n/\Inn(F_n)$,  
coincides with the outer Torelli group, 
$\wT_{F_n}=\ker\, (\Out(F_n) \surj \GL_n(\Z))$. 
We will denote the induced Johnson filtration on $\Out(F_n)$ 
by $\widetilde{J}^{s}_{n}=\pi(J^s_n)$.  

The groups defined so far in this section fit into the following 
commuting diagram, with exact rows and columns:
\begin{equation}
\label{eq:aut fn}
\xymatrixrowsep{16pt}
\xymatrix{
     &\Inn(F_n) \ar^{=}[r] \ar@{^{(}->}[d]&\Inn(F_n) \ar@{^{(}->}[d]\\
1\ar[r] &\IA_n \ar[r] \ar@{->>}^{\pi}[d]&\Aut(F_n) \ar@{->>}^{\pi}[d] \ar[r] 
& \GL_n(\Z) \ar^{=}[d] \ar[r] & 1\\
1\ar[r] &\OA_n \ar[r] &\Out(F_n) \ar[r] & \GL_n(\Z) \ar[r] & 1
}
\end{equation}

In \cite{Ma}, Magnus showed that the group $\IA_n$ has 
finite generating set 
\begin{equation}
\label{eq:ian gens}
\{\alpha_{ij}, \alpha_{ijk} \mid 1\le i\ne j\ne k \le n \},
\end{equation}
where $\alpha_{ij}$ maps $x_i$ to $x_j x_ix_j^{-1}$ 
and $\alpha_{ijk}$ maps  $x_i$ to $x_i(x_j,x_k)$, with both 
$\alpha_{ij}$ and $\alpha_{ijk}$ fixing the remaining generators.
In particular, $\IA_1=\{1\}$ and $\IA_2=\Inn(F_2)\cong F_2$ 
are finitely presented. 
On the other hand,  Krsti\'{c} and McCool showed in \cite{KM} 
that $\IA_3$ is not finitely presentable.  It is still unknown 
whether $\IA_n$ admits a finite presentation for $n\ge 4$.

\subsection{Braid-like subgroups}
\label{subsec:mccool}

In \cite{McC}, J.~McCool identified a remarkable subgroup 
of $\IA_n$, consisting of those automorphisms of $F_n$ 
which send each generator $x_i$ to a conjugate of itself.  
It turns out that McCool's group of ``pure symmetric" automorphisms, 
$\PS_n$, is precisely the subgroup generated by the Magnus 
automorphisms $\{ \alpha_{ij} \mid 1\le i\ne j \le n\}$; furthermore, 
$\PS_n$ is finitely presented.  

The ``upper triangular" McCool group, denoted $\PS_n^{+}$, is the 
subgroup generated by the automorphisms $\alpha_{ij}$ with $i> j$.  
Let $K_n$ be the subgroup generated by the automorphisms 
$\alpha_{nj}$, with $1\le j\le n-1$.  
In  \cite{CPVW},  Cohen, Pakianathan, Vershinin, and Wu show that 
$K_n$ is isomorphic to the free group $F_{n-1}$ on these generators.
They also show that $K_n$ is a normal subgroup of $\PS_n^{+}$,  
with quotient group $\PS_{n-1}^{+}$.  Moreover, the sequence 
\begin{equation}
\label{eq:psplus}
\xymatrix{
1\ar[r] &K_n\ar[r] &\PS_n^{+} \ar[r] & \PS_{n-1}^{+} \ar[r] & 1
}
\end{equation}
is split exact, with $\PS_{n-1}^{+}$ acting on $K_n\cong F_{n-1}$ by 
$\IA$-automorphisms. It follows that the upper triangular McCool group
has the structure of an iterated semidirect product of free groups,
$\PS_n^{+} = F_{n-1} \rtimes \cdots \rtimes F_2 \rtimes F_1$, 
with all extensions given by $\IA$-automorphisms. 

An analogous decomposition holds for the Artin pure braid group 
$P_n$, which is the group consisting of those automorphisms in $\PS_n$ that 
leave the word $x_1\cdots x_n\in F_n$ invariant.  
Again, there are split epimorphisms $p_n\colon P_n\surj P_{n-1}$, 
with $\ker(p_n)\cong F_{n-1}$, leading to iterated semidirect product 
decompositions, $P_n = F_{n-1} \rtimes \cdots \rtimes F_2 \rtimes F_1$, 
with all extensions given by pure braid automorphisms. 

\begin{prop}
\label{prop:bcg}
Let $G$ be one of the groups $\PS_n$,  $\PS_n^{+}$, or $P_n$, 
with $n\ge 3$.  Then 
\begin{enumerate}
\item \label{cr1} The group $G$ is very large.
\item \label{cr2} $\dim B_{\C}(G) =\infty$. 
\end{enumerate}
\end{prop}

\begin{proof}
First consider the case when $G=P_n$ or $G=\PS_n^{+}$. 
By the above discussion, there is an epimorphism 
from $G$ to a group of the form $F_2\rtimes_{\alpha} F_1$, 
where $\alpha$ belongs to $\IA_2$. But, as noted in 
\S\ref{subsec:fn}, $\IA_2=\Inn(F_2)$.  Thus, 
$F_2 \rtimes_{\alpha} F_1 \cong F_2\times F_1$, 
and so $G$ admits an epimorphism onto $F_2$.

As for the full McCool groups, 
we know from \cite{CPVW} that there are epimorphisms 
$\PS_n \surj \PS_{n-1}$.  Since $\PS_2\cong F_2$ is freely 
generated by $\alpha_{12}$ and $\alpha_{21}$, we conclude
that $\PS_n$ is very large as well. 

Part \eqref{cr2} now follows from part \eqref{cr1} and 
Proposition \ref{prop:cork}.
\end{proof}

As we shall see in Theorem \ref{thm:fhp}, the analogue of 
Proposition \ref{prop:bcg} does not hold for the ambient group, 
$G=\IA_n$, as soon as $n\ge 5$. 

\subsection{The Johnson homomorphisms}
\label{subsec:jfn}

We now record several morphisms between the various graded 
Lie algebras under consideration.  To avoid trivialities, we will 
henceforth assume that $n\ge 3$.

\begin{itemize}
\item
Let $J\colon \gr_F (\IA_n) \to \Der (\L_n)$ be the Johnson 
homomorphism.  In view of Theorems \ref{thm:outer johnson} 
and \ref{thm:free lie}, we also have an outer 
Johnson homomorphism,   
$\widetilde{J}\colon \gr_{\wF} (\OA_n) \to \widetilde{\Der} (\L_n)$.  

\item
Let $\kappa\colon K_n \to \IA_n$ be the inclusion map, 
given by the chain of subgroups $K_n< \PS_{n}^{+} 
< \PS_{n} < \IA_n$, and let 
$\gr_{\Gamma}(\kappa)\colon \gr_{\Gamma} (K_n) 
\to \gr_{\Gamma} (\IA_n)$ be the induced morphism.  

\item
Given an element $y\in \L^s_n$, let 
$\ev_y \colon \Der (\L_n) \to \L_n$ be the (degree $s$)
evaluation map, given by $\ev_y(\delta)=\delta(y)$. 
In particular, we have $\Z$-linear maps
$\ev_i:=\ev_{\bar{x}_i}\colon \Der^{*} (\L_n) \to \L^{*+1}_n$,
between the respective graded, torsion-free abelian groups.
\end{itemize}

Inserting these morphisms, as well as the $\iota$-maps 
from \eqref{eq:iota phi} in diagram \eqref{eq:jcd}, we obtain the  
following commuting diagram:
\begin{equation}
\label{eq:jpsi}
\xymatrixcolsep{7pt}
\xymatrixrowsep{20pt}
\xymatrix{
&&&& \gr_{\Gamma} (F_n)\cong \L_n 
\ar@{_{(}->}_{\overline{\Ad}}[ld] \ar@{^{(}->}^{\ad}[rd]
\\
&\gr_{\Gamma} (\IA_n) \ar^{\iota_F}[rr] \ar@{->>}^{\gr_{\Gamma}(\pi)}[dd]
&&\gr_{F} (\IA_n) \ar^{J}[rr]  \ar@{->>}^(.46){\overline{\pi}}[dd]
&& \Der (\L_n) \ar@{->>}^{q}[dd]  \ar^(.61){\ev_{i}}[rr] 
&& \L_n
\\
\L_{n-1}\cong \gr_{\Gamma} (K_n) \quad
\ar^{\gr_{\Gamma}(\kappa)}[ur] 
\ar^{\gr_{\Gamma}(\pi\circ \kappa)}[dr]   
\ar@{-->}@/_/^(.68){\psi}[urrrrr]
\ar@{-->}@/^/^(.68){\tilde\psi}[drrrrr]
\\
&\gr_{\Gamma} (\OA_n) \ar^{\iota_{\wF}}[rr] 
&&\gr_{\wF} (\OA_n) \ar^{\widetilde{J}}[rr]  
&& \widetilde{\Der} (\L_n)
}
\end{equation}

Consider the composite $\psi=J\circ \iota_F\circ \gr_{\Gamma}(\kappa)$, 
marked by the top dotted arrow in the diagram.  We need two 
technical lemmas regarding this morphism.

\begin{lemma}[\cite{CHP}, Proposition 6.2]
\label{lem:ev psi} 
For all $n$ and $s$, the restriction of $\ev_i \circ \psi$ to 
$\L_{n-1}^s$ equals
\[
\begin{cases}
0 & \text {if\, $1\le i \le n-1$},\\[2pt]
(-1)^s \ad_{\bar{x}_n} & \text{if\, $i$}.
\end{cases}
\]
\end{lemma}

The sign $(-1)^s$ appears in the above because the authors of 
\cite{CHP} work with the group commutator $x^{-1}y^{-1}xy$. 

\begin{lemma}
\label{lem:im j} 
$\im(\psi) \cap \im (\ad) =\{0\}$.
\end{lemma}

\begin{proof}
It is enough to consider homogeneous elements of a fixed 
degree $s\ge 1$.  So suppose $\psi (\bar{w})=\ad_{\bar{u}}$, 
for some $\bar{w}\in \L^s_{n-1}$ and $\bar{u}\in \L^s_{n}$. 
By Lemma \ref{lem:ev psi}, we have 
$\ev_{i} (\psi (\bar{w}))=0$, for each $1\le i\le n-1$. 
Therefore, $0=\ev_{i} (\ad_{\bar{u}})=[\bar{u}, \bar{x}_i]$. 
Since the Lie algebra $\L_n$ is free, $\bar{u}$ and $\bar{x}_i$ 
must be linearly dependent over $\Z$,  for each $1\le i\le n-1$ 
(cf.~\cite{MKS}). 
Since $n\ge 3$, this forces $\bar{u}=0$, and we are done. 
\end{proof}

\subsection{Associated graded Lie algebra and Alexander invariant}
\label{subsec:gr ia}

We are now in a position to prove the main result of this section. 

\begin{theorem}
\label{thm:gr ia} 
For each $n\ge 3$, the $\Q$-vector spaces $\gr_{\Gamma}(\IA_n) \otimes \Q$ 
and $\gr_{\Gamma}(\OA_n) \otimes \Q$ are infinite-dimensional.
\end{theorem}

\begin{proof}
By Lemma \ref{lem:ev psi}, the evaluation map 
$\ev_n \colon \Der (\L_n)\to \L_n$ restricts to
a surjective, $\Z$-linear map 
\begin{equation}
\label{eq:evnl}
\xymatrix{\ev_n \colon \im (\psi) \ar@{->>}[r]& 
[\bar{x}_n, \L (\bar{x}_1, \dots, \bar{x}_{n-1})] \cong \L_{n-1}}.
\end{equation}
Since $n\ge 3$, Theorem \ref{thm:free lie} insures that 
$\dim_{\Q} \L_{n-1} \otimes \Q= \infty$. Hence, 
$\rank \im(\psi)=\infty$, and thus 
$\dim_{\Q} \gr_{\Gamma}(\IA_n) \otimes \Q=\infty$.  

Now let $\tilde{\psi}$ 
be the homomorphism indicated by the bottom dotted arrow 
in diagram \eqref{eq:jpsi}. In view of  Lemma \ref{lem:im j}, 
and the fact that $\ker(q)=\im(\ad)$, the map $q$ restricts 
to an isomorphism $q\colon \im(\psi) \isom  \im(\tilde{\psi})$. 
Hence, 
$\rank \im(\tilde{\psi})=\infty$, and thus 
$\dim_{\Q} \gr_{\Gamma}(\OA_n) \otimes \Q=\infty$.   
\end{proof}

\begin{corollary}
\label{cor:gr ia} 
Let $G$ be either $\IA_n$ or $\OA_n$, and assume 
$n\ge 3$.  Then the $\Q{G}_{\ab}$-module $B_{\Q}(G)$ does 
{\em not}\/ have trivial $G_{\ab}$-action.
\end{corollary}

\begin{proof}
By Proposition \ref{prop:alex gr}, triviality of $B_{\Q}(G)$ 
implies $\gr_{\G}^3 (G)\otimes \Q=0$. 
Therefore, $\gr_{\G}^q (G)\otimes \Q=0$ for all $q\ge 3$, 
since the associated graded Lie algebra is generated in 
degree one. In particular, 
$\dim_{\Q} \gr_{\G} (G)\otimes \Q <\infty$, contradicting 
Theorem \ref{thm:gr ia}.
\end{proof}

\section{Arithmetic group symmetry and the first resonance variety}
\label{sec:coho}

In this section we exploit the natural $\GL_n(\Z)$-action on the cohomology 
ring of $\OA_n$ to show that the first resonance variety of this group 
vanishes, in the range $n\ge 4$. 

\subsection{Structure of the abelianization}
\label{subsec:coho}

We apply the machinery developed in Sections \ref{sect:johnson}
and \ref{sect:outer johnson} to the group $G=F_n$.   Recall that 
in \S\ref{subsec:fn} we identified  the symmetry group $\A(F_n)$ with the 
general linear group $\GL_n(\Z)$.  The conjugation action of this group  
on the kernels $\IA_n$ and $\OA_n$ from the two exact rows in \eqref{eq:aut fn} 
induces $\GL_n (\Z)$ actions by graded Lie algebra automorphisms on 
both $\gr_{\G} (\IA_n)$ and $\gr_{\G} (\OA_n)$.  

Denote by $H$ the free abelian group $H_1 (F_n, \Z)=\Z^n$, 
viewed as a $\GL_n (\Z)$-module via the defining representation. 
Using Theorem \ref{thm:free lie}, we may identify the associated graded 
Lie algebra $\gr_{\G}(F_n)$ with the free Lie algebra $\L_n=\Lie(H)$.
It is now a standard exercise to identify 
$\Der^1 (\L_n)=H^* \otimes  (H\wedge H)$ and 
$\widetilde{\Der}^1 (\L_n)=H^* \otimes (H\wedge H)/H$. 
With these identifications at hand, the exact sequence 
\eqref{eq:outder} for $\g=\L_n$, in degree $1$, takes the form
\begin{equation}
\label{eq:outder fn}
\xymatrix{
0\ar[r] &H \ar^(.3){\ad}[r] &H^* \otimes  (H\wedge H) \ar[r] 
&H^* \otimes (H\wedge H)/H \ar[r] & 0
}.
\end{equation}

Let $\{ e_1, \dots, e_n\}$ be the standard $\Z$-basis of $H$, with
dual basis denoted $\{ e_1^*, \dots, e_n^*\}$.  By construction,
$\ad (e_i)= \sum_{j=1}^n e_j^* \otimes (e_i \wedge e_j)$, for all $i$.
It follows that the exact sequence \eqref{eq:outder} is split exact 
in the case $\g=\L_n$. In particular, all the terms in sequence 
\eqref{eq:outder fn} are (finitely generated) free abelian groups.

From sections \ref{subsec:jfilt} 
and \ref{subsec:out torelli}, we know that the symmetry group 
$\A(F_n)=\GL_n(\Z)$ naturally acts on the abelianizations 
$(\IA_n)_{\ab}=\gr_{\G}^1(\IA_n)$ and $(\OA_n)_{\ab}=\gr_{\G}^1(\OA_n)$. 
These groups, viewed as $\GL_n(\Z)$-rep\-resentation spaces, 
were computed by Andreadakis \cite{An} for $n=3$, 
and by F.~Cohen and J.~Pakianathan, B.~Farb, and 
N.~Kawazumi in general.  The next theorem summarizes 
those  results, essentially in the form given by Pettet in \cite{Pe}.

\begin{theorem}
\label{thm:pettet}
For each $n\ge 3$, the maps 
\begin{equation*}
\xymatrixrowsep{2pt}
\xymatrix{
J\circ \iota_F \colon (\IA_n)_{\ab} \ar[r] & H^* \otimes (H\wedge  H)
\\
\wJ \circ \iota_{\wF} \colon (\OA_n)_{\ab} \ar[r] & H^* \otimes (H\wedge H) /H
}
\end{equation*}
are $\GL_n(\Z)$-equivariant isomorphisms.
\end{theorem}

As an application of this theorem, and of the general theory 
from \S\S\ref{sect:johnson}--\ref{sect:outer johnson}, 
we derive the following consequence.

\begin{corollary}
\label{cor:jprime}
For each $n\ge 3$, the following hold. 
\begin{enumerate}
\item \label{jo1}
$J_n^2 =\IA_n'$ and $\wJ_n^2 =\OA_n'$. 
\item \label{jo2}
We have an exact sequence
\begin{equation*}
\label{eq:iaseq}
\xymatrix{1\ar[r] & F_n' \ar^(.43){\Ad}[r] 
& \IA_n' \ar[r] & \OA_n' \ar[r] & 1},
\end{equation*}
with $\OA_n'$ acting on $(F_n')_{\ab}$ by $\bar{\alpha}\cdot \bar{x}= 
\overline{\alpha(x)}$, for all $\alpha\in \IA_n' $ and $x\in F_n'$. 

\end{enumerate}
\end{corollary}

\begin{proof}
Theorems \ref{thm:free lie} and \ref{thm:pettet} guarantee that 
all the hypotheses of Corollary 
\ref{cor:outer johnson2} are satisfied for the group $G=F_n$. 
The desired conclusions follow at once.
\end{proof}

Set $U:=H^* \otimes (H\wedge  H)$ and $L:=H^* \otimes (H\wedge H) /H$. 
Plainly, both $\GL_n(\Z)$-representations, $U\otimes \C$ and $L\otimes \C$, 
extend to rational representations of $\GL_n (\C)$. 
As shown by Pettet in \cite{Pe}, the $\GL_n (\C)$-representation 
$L\otimes \C$ is irreducible, while $U\otimes \C$ is not. 
Because of this, 
we will focus our attention for the rest of this section on the group $\OA_n$. 

\subsection{Cohomology and $\sl_n (\C)$-representation spaces}
\label{subsec:rep}

Consider now the cup-product map in the low-degree 
cohomology of $\OA_n$,
\begin{equation}
\label{eq:cup}
\cup_{\OA_n} \colon H^1 (\OA_n, \C) \wedge H^1 (\OA_n, \C) \to H^2 (\OA_n, \C),
\end{equation} 
and set 
\begin{equation}
\label{eq:v and k}
V= H^1 (\OA_n, \C),  \qquad 
K=\ker (\cup_{\OA_n} ). 
\end{equation}
By Theorem \ref{thm:pettet}, we have $V= (L\otimes \C)^*= 
H_{\C} \otimes (H_{\C}^* \wedge H_{\C}^*) /H_{\C}^*$, where 
$H_{\C}=H_1(F_n,\C)=\C^n$.

It follows from the general setup discussed in \cite{DP} that 
the vector space $K\subset V\wedge V$ is $\SL_n(\C)$-invariant. 
Let $\sl_n (\C)$ be the Lie  algebra of $\SL_n(\C)$.  The work 
of Pettet \cite{Pe} determines explicitly the infinitesimal 
$\sl_n (\C)$-representations associated to $V$ and $K$.

To explain this result, let us first review some basic notions 
from the representation theory of $\sl_n (\C)$ and $\gl_n (\C)$.  
Following the setup in Fulton and Harris' book \cite{FH}, denote 
by $\{ t_1,\dots, t_n \}$ the dual coordinates of
the diagonal matrices from $\gl_n (\C)$, and set 
$\lambda_i =t_1+ \cdots+t_i$, for $i=1,\dots,n$.
Let $\dd_n$ be the standard diagonal Cartan subalgebra of $\sl_n (\C)$. 
Then
\begin{equation}
\label{eq:diagonal}
\dd_n^* =\C\text{-span}\, \{ t_1,\dots, t_n \}/ \C\cdot \lambda_n. 
\end{equation}

The strictly upper-triangular matrices in $\sl_n (\C)$ are 
denoted by $\sl_n^{+} (\C)$.  The corresponding set of positive 
roots is $\Phi_n^+= \{ t_i-t_j \mid 1\le i<j \le n \}$. 
Finally, the finite-dimensional irreducible representations of $\sl_n (\C)$ 
are parametrized by tuples $\mathbf{a}= (a_1,\dots, a_{n-1})\in \N^{n-1}$.   
To such a tuple $\mathbf{a}$, there corresponds an irreducible representation 
$V(\lambda)$, with highest weight $\lambda =\sum_{i<n} a_i \lambda_i$. 

\begin{theorem}[\cite{Pe}]
\label{thm:pet}
Fix $n\ge 4$, and set $\lambda= \lambda_1 +\lambda_{n-2}$ 
and $\mu= \lambda_1 +\lambda_{n-2} +\lambda_{n-1}$.  Then 
$V=V(\lambda)$ and $K=V(\mu)$, as $\sl_n (\C)$-modules.  
\end{theorem}

\subsection{A maximal vector for $V$}
\label{subsec:max vect}

We now find a highest weight vector $v_0$ for the $\sl_n(\C)$-repre\-sentation 
space $V=V(\lambda)$.  Let $\{ e_1,\dots, e_n \}$ be the standard 
basis of the vector space $H_{\C}= H_1(F_n, \C)= \C^n$, 
and denote by $\{ e_1^*,\dots, e_n^* \}$ the dual basis of $H_{\C}^*$. 
From \S\S\ref{subsec:coho}--\ref{subsec:rep}, we know that 
\begin{equation}
\label{eq:h1oa}
H_1(\OA_n, \C)= \coker ( \ad\colon  
H_{\C} \to H_{\C}^* \otimes H_{\C} \wedge H_{\C}),
\end{equation}
where $\ad (e_i)= \sum_{j=1}^n e_j^* \otimes e_i \wedge e_j$. 
It follows that 
\begin{equation}
\label{eq:v again}
V= \ker\, ( \ad^*\colon   H_{\C} \otimes H_{\C}^*\wedge H_{\C}^* \to H_{\C}^* ),
\end{equation}
where the linear map $\ad^*$ is given by
\begin{equation}
\label{eq:iota star}
\ad^*(e_i\otimes e_j^* \wedge e_k^*)=
\begin{cases}
-e_k^* & \text{if $i=j$},\\
0 & \text{if $\abs{\{i,j,k\}}=3$}.
\end{cases}
\end{equation}

We also know from Theorem \ref{thm:pet} that $V=V(\lambda)$ 
as $\sl_n(\C)$-modules, where $\lambda= t_1-t_{n-1}-t_n$. 
Note that, for each $x\in \sl_n(\C)$, we have
\begin{equation}
\label{eq:xcdot}
x\cdot (e_i\otimes e_j^* \wedge e_k^*) =x\cdot e_i\otimes e_j^* \wedge e_k^* -
e_i\otimes e_j^* \circ x \wedge e_k^* -e_i\otimes e_j^* \wedge e_k^* \circ x.
\end{equation}

For each pair of distinct integers $l,m\in [n]$, denote by $x_{lm}$ 
the endomorphism of $H_{\C}$ 
sending $e_m$ to $e_l$, and $e_p$ to $0$ for $p\ne m$. 
Then $\{x_{lm} \}_{1\le l<m\le n}$ is a $\C$-basis of $\sl_n^+(\C)$.  

\begin{lemma}
\label{lem:v0}
The element $v_0=e_1\otimes (e_{n-1}^*\wedge e_n^*)$ is a 
maximal vector for $V=V(\lambda)$. 
\end{lemma}

\begin{proof}
Clearly, $v_0\ne 0$. We need to verify the following facts:
\begin{enumerate}
\item \label{i} 
$\ad^*(v_0)=0$.
\item \label{ii} 
$\sl_n^+(\C)\cdot v_0=0$.
\item \label{iii} 
$\weight(v_0)=\lambda$.
\end{enumerate}
All three statements are checked by straightforward direct computation.
\end{proof}

\subsection{Symmetry of resonance varieties}
\label{subsec:sym res}

We now recall a result from \cite{DP}, which will provide a 
key representation-theoretic tool for computing the resonance 
varieties $\RR (\OA_n)$, for $n\ge 4$.  

\begin{lemma}[\cite{DP}]
\label{lem:tor32}
Let $S$ be a semisimple, linear algebraic group over $\C$, with 
standard decomposition of its Lie algebra, $\ss= \ss^- +\dd + \ss^+$. 
Let $B\subseteq S$ be the associated Borel subgroup, with 
Lie algebra $\dd + \ss^+$. Let $V$ be an irreducible, rational 
$S$-representation, and consider a Zariski closed, $S$-invariant 
cone, $\RR \subseteq V$. If $\RR \ne \{ 0\}$, then 
\begin{enumerate}
\item \label{borel1}
$B$ has a fixed point in the projectivization $\PP(\RR) \subseteq \PP (V)$. 
\item \label{borel2} 
$\RR$ contains a maximal vector of the $\ss$-module $V$. 
\end{enumerate}
\end{lemma}

\begin{proof}
The first assertion is a direct consequence of Borel's fixed point 
theorem \cite{Hu75}. The argument from \cite[Lemma 3.2]{DP} 
proves the second assertion.
\end{proof}

We will apply this general result to the algebraic group $S=\SL_n(\C)$, 
with Lie algebra $\ss=\sl_n(\C)$. 
Denote by $B\subseteq \SL_n (\C)$ the closed, connected, solvable 
subgroup consisting of all upper-triangular matrices from $\SL_n (\C)$. 
Then $B$ is the Borel subgroup containing the standard 
diagonal maximal torus, with Lie algebra $\dd_n +\sl_n^+(\C)$.

Let $V=V(\lambda)$ be the irreducible representation discussed 
previously.   Clearly, the co-restriction of the cup-product 
map, $\cup =\cup_{\OA_n}\colon V\wedge V \to V \wedge V/K$, 
is $S$-equivariant.  Thus, the resonance variety $\RR =\RR (\OA_n) \subseteq V$, 
as well as its projectivization, $\PP(\RR) \subseteq \PP (V)$,
inherit in a natural way an $S$-action.   Therefore, all the basic 
premises of Lemma \ref{lem:tor32} are met in our situation.

\subsection{Vanishing resonance}
\label{ss83}
Before proceeding, we need one more technical lemma. 
Recall that $v_0$ is a maximal  vector for the 
irreducible $\sl_n(\C)$-module $V=V(\lambda)$; 
in particular, $v_0\ne 0$. 
Let $K\subset V \wedge V$ be the kernel of the 
cup-product map $\cup=\cup_{\OA_n}$.  
By Theorem \ref{thm:pet}, we have that $K=V(\mu)$, where  
$\mu=\lambda -t_n$. Let $u_0\in K$ be a maximal vector 
for this irreducible $\sl_n(\C)$-module.

\begin{lemma}
\label{lem:lem3.5}
Suppose $v_0\in \RR (\OA_n)$. There is then an element $w\in V$ 
(of weight $\mu-\lambda=- t_n$) such that $u_0=v_0\wedge w \ne 0$. 
\end{lemma}

\begin{proof}
Let $\ell_0\colon V\to V \wedge V$ denote left-multiplication 
by $v_0$ in the exterior algebra $\bigwedge V$. By definition 
of resonance, the vector $v_0$ belongs to $\RR(\OA_n)$ if and 
only if $\im (\ell_0)\cap K \ne 0$. 
An argument similar to the one from \cite[Lemma 3.5]{DP} yields the 
desired element $w\in V$. For the reader's convenience, we include 
the details of the proof.

The Lie algebra $\sl_n^+ (\C)$ decomposes into a direct sum 
of $1$-dimensional vector spaces, each one spanned by a matrix 
$x_{\alpha}$, with $\alpha$ running through the set of positive roots, 
$\Phi^{+}_n$. Let $K_{\nu}$ be a non-trivial weight 
space of $K=V(\mu)$, and let $\alpha_1,\dots ,\alpha_r\in \Phi^{+}_n$.  
Then 
$x_{\alpha_1}\cdots x_{\alpha_r}  (K_{\nu}) \subseteq  K_{\nu'}$, 
where $\nu'=\nu + \sum_{i=1}^r \alpha_i$.  If $\nu'\ne 0$, 
then, by maximality of $\mu$, we can write $\nu'=\mu-\beta$, 
where $\beta$ is a linear combination of elements in $\Phi_n^{+}$, 
with coefficients in $\Z_{>0}$, see \cite[Theorem 20.2(b)]{H72}. 
Choosing an integer $r$ so that $r\ge \height(\mu-\nu)$ guarantees 
that $K_{\nu'}=0$.  It follows that each element of $\sl_n^+ (\C)$ 
acts nilpotently on $K$.

From Lemma \ref{lem:v0}\eqref{ii}, we know that the 
Lie algebra $\sl_n^+ (\C)$ annihilates the vector $v_0$; 
hence, the linear map $\ell_0$ is $\sl_n^+ (\C)$-equivariant. 
Since $v_0$ belongs to the weight subspace $V_{\lambda}$,
it follows that 
$\ell_0 (V_{\lambda'})\subseteq (V\wedge V)_{\lambda +\lambda'}$,
for each weight subspace $V_{\lambda'}\subseteq V$. 
Applying Engel's theorem (cf.~\cite[Theorem 3.3]{H72}), 
we infer that the $\sl_n^+ (\C)$-module $\im (\ell_0)\cap K$ 
contains a non-zero vector annihilated by $\sl_n^+ (\C)$. 
Since this vector has weight $\mu$, it must be equal (up 
to a non-zero scalar) to the maximal vector $u_0$. 
Therefore, $u_0=\ell_0(w)$, for some $w\in V_{\mu -\lambda}$, 
and we are done.
\end{proof}

We are now ready to state and prove the main result of this section.

\begin{theorem}
\label{thm:res oan}
For all $n\ge 4$, 
\[
\RR(\OA_n)=\{0\}.
\]
\end{theorem}

\begin{proof}
Suppose that $\RR(\OA_n) \ne \{ 0\}$.  Then, by  
Lemma \ref{lem:tor32}, we must have $v_0\in \RR(\OA_n)$. 
Thus, by Lemma \ref{lem:lem3.5}, there is an element $w\in V$ 
(of weight $- t_n$) such that $u_0=v_0\wedge w$. 
Since $\weight (w)=-t_n$, we have
\begin{equation}
\label{eq:w}
w=\sum_{i=1}^{n-1} c_i e_i  \otimes e_i^* \wedge e_n^*,
\end{equation}
for some $c_i\in \C$.  To finish the proof, it is enough to show 
that all coefficients $c_i$ vanish, which will contradict the fact 
that $u_0=v_0\wedge w\ne 0$.

Applying $\ad^*$ to both sides of \eqref{eq:w}, we get 
$0=-\big(\sum_{i<n} c_i\big) e_n^*$, and thus 
\begin{equation}
\label{eq:sum c}
\sum_{i<n} c_i = 0.
\end{equation}

Now, since $u_0$ is a maximal vector for $K=V(\mu)$, we 
must have $xu_0=0$, for every $x\in \sl_n^+(\C)$.  Since 
$u_0=v_0\wedge w$, we have $xu_0=xv_0\wedge w 
+ v_0 \wedge xw$. On the other hand, since  $v_0$ is a maximal 
vector for $V(\lambda)$, we also have $xv_0=0$. Hence,
$v_0\wedge x w=0$, and we conclude that
\begin{equation}
\label{eq:xw}
xw\in \C\cdot v_0.
\end{equation}

Now take $x=x_{jk}$, 
where $1\le j<k<n$.  A simple computation shows that
\begin{equation}
\label{eq:xjk}
x_{jk}\cdot (e_i\otimes e_i^* \wedge e_n^*) = 
\begin{cases}
-e_j\otimes e_k^* \wedge e_n^* &\quad\text{if $i=j$}, \\
e_j\otimes e_k^* \wedge e_n^* &\quad\text{if $i=k$}, \\
0 &\quad\text{otherwise}.
\end{cases}
\end{equation}
Therefore:
\begin{equation}
\label{eq:xjk w}
x_{jk} w=(c_k-c_j) e_j\otimes e_k^* \wedge e_n^*.
\end{equation}

Finally, pick a pair $(j,k)\ne (1,n-1)$. Using 
\eqref{eq:xw} and \eqref{eq:xjk w}, we get:
\begin{equation}
\label{eq:ckj}
c_k-c_j=0.
\end{equation}
Using the fact that $n\ge 4$, we obtain:
\begin{equation}
\label{eq:cn1}
c_1=c_2=\cdots =c_{n-1}.
\end{equation}
Equation \eqref{eq:sum c} now yields $c_i=0$, for all $i<n$, 
and we are done. 
\end{proof}

In view of Theorem \ref{prop:res bg}, the above result has 
some immediate consequences pertaining to the $I$-adic 
completion of the complexified Alexander invariant $B_{\C}(\OA_n)$, 
and the infinitesimal Alexander invariant $\B(\OA_n)$.

\begin{corollary}
\label{cor:boan}
For $n\ge 4$, the vector spaces $\widehat{B_{\C}(\OA_n)}$ 
and $\B(\OA_n)$ are finite-dimensional.
\end{corollary}

\section{Homological finiteness}
\label{sect:oian}

In this final section, we show that the first Betti numbers 
of $\OA_n'$ and $\IA_n'$ are finite, and derive some consequences. 

\subsection{The first Betti number of $\OA_n'$}
\label{subsec:b1oan}

We start with the outer automorphism group.  Before proceeding, 
we need to recall a general result from \cite{DP}. 
Let $S$ be a complex, simple linear algebraic group defined over $\Q$, 
with $\Q$-$\rank (S)\ge 1$, and let $D$ be an arithmetic subgroup 
of $S$.  

\begin{theorem}[\cite{DP}]
\label{thm:B}
Suppose $D$ acts on a free, finitely generated abelian  
group $L$, such that the $D$-action on $L\otimes \C$ extends to 
a rational, irreducible $S$-representation. Then, the associated $D$-action
on $\TT (L)$ is {\em geometrically irreducible}, i.e., the only 
$D$-invariant, Zariski closed subsets of $\TT (L)$
are either equal to $\TT (L)$, or finite.
\end{theorem}

Now note that the discrete group $D=\SL_n (\Z)$ is an arithmetic 
subgroup of the linear algebraic group $S=\SL_n (\C)$.  Moreover, 
$S$ is a simple group defined over the rationals, and 
the $\Q$-rank of $S$ equals $n-1$.  Thus, the pair 
$(S,D)$ satisfies all requirements needed in the 
above theorem.

\begin{theorem}
\label{thm:a}
For $n\ge 4$, the following hold.
\begin{enumerate}
\item \label{tai1}
The first characteristic variety $\VV (\OA_n)$ is finite.
\item \label{tai2}
The Alexander polynomial of $\OA_n$ is a non-zero constant, modulo units.
\item \label{tai3}
If $N$ is a subgroup of $\OA_n$ containing $\OA'_n$, then $b_1(N)<\infty$.
\end{enumerate}
\end{theorem}

\begin{proof}
\eqref{tai1}
From Theorem \ref{thm:pettet}, we know that the abelian group 
$L=(\OA_n)_{\ab}$ is torsion-free (of finite rank). Note that the $D$-action 
on $L$ coming from $\Out (F_n)$-conjugation in the exact sequence 
\begin{equation}
\label{eq:oaout}
\xymatrix{1\ar[r] & \OA_n \ar[r] & \Out (F_n) \ar[r] & \GL_n (\Z) \ar[r] & 1}
\end{equation}
extends to an irreducible, rational $S$-representation in $L\otimes \C$. 

Consider the $D$-action on the character torus $\TT (\OA_n)= \Hom (L, \C^{\times})$,
associated to the above $D$-representation in $L$.
In view of \cite[Lemma 5.1]{DP}, the characteristic variety $\VV (\OA_n)$ 
is a $D$-invariant, Zariski closed subset of $\TT (\OA_n)$. 
Furthermore, by Theorem \ref{thm:B}, the affine torus 
$\TT (\OA_n)$ is geometrically $D$-irreducible.  

Now suppose $\VV (\OA_n)$ is infinite.  The above two facts 
would then force $\VV (\OA_n)= \TT (\OA_n)$. 
Theorem \ref{thm:tcone}\eqref{t1} would then imply that 
$\RR (\OA_n)=H^1 (\OA_n, \C)$, 
contradicting Theorem \ref{thm:res oan}.

\eqref{tai2}
By the above, $\VV (\OA_n)$ is finite. Since $b_1(\OA_n)\ge 2$,  
the desired conclusion follows from Corollary \ref{cor:delta const}.

\eqref{tai3} Consider the exact sequence
\begin{equation}
\label{eq:oan}
\xymatrix{1\ar[r] & \OA_n' \ar[r] & N \ar[r] & N/\OA_n' \ar[r] & 1} .
\end{equation}
The quotient group $N/\OA_n'$ is a subgroup of $L=\OA_n/\OA_n'$; 
hence $b_1 (N/\OA_n')<\infty$.  Using Part \eqref{tai1} and 
Corollary \ref{cor:btest}\eqref{bt2}, we find that 
$b_1 (\OA_n')<\infty$. Finally, a standard application 
of the Hochschild-Serre spectral sequence for the 
extension \eqref{eq:oan} shows that $b_1(N)<\infty$.
\end{proof} 

\subsection{A nilpotence lemma}
\label{subsec:nilp2}
 
Let $R$ be the group ring of $\Z^n=(F_n)_{\ab}$, and let 
$B(F_n)=F_n'/F_n''$ be the Alexander invariant of $F_n$, 
viewed as an $R$-module, as in Example \ref{ex:bfree}.   
From Proposition \ref{prop:rlin} and the discussion preceding 
it, we know that the Torelli group $\IA_n=\T_{F_n}$ acts 
$R$-linearly on the module $B(F_n)$, via 
$\alpha\cdot \bar{x}=\overline{\alpha(x)}$.   

Restrict now this action to the derived subgroup of $\IA_n$, 
and consider the module of coinvariants, denoted $B(F_n)_{\IA_n'}$.  
By definition, this is the quotient of $B(F_n)$ by the $\Z$-submodule 
generated by all elements of the form $\bar{x} - \alpha\cdot \bar{x}$, 
with $x\in F_n'$ and $\alpha\in \IA_n'$. Since $\IA_n'$ acts $R$-linearly, 
this $\Z$-submodule is actually an $R$-submodule.  Thus, the 
canonical projection map,
\begin{equation}
\label{eq:linp}
\xymatrix{p\colon B(F_n) \ar@{->>}[r]& B(F_n)_{\IA_n'}}, 
\end{equation}
is an $R$-linear map.  The next lemma shows that the quotient module 
is (two-step) nilpotent, provided $n$ is sufficiently large.

\begin{lemma}
\label{lem:isq}
Let $I$ be the augmentation ideal of $R$.  
If $n\ge 5$, then 
\[
I^2\cdot B(F_n)_{\IA_n'} =0.
\]
\end{lemma}

\begin{proof}
Let $x_1,\dots, x_n$ be free generators for the group $F_n$. 
As is well-known, $I$ is generated (as an $R$-module) 
by all elements of the form $x_k-1$; thus, $I^2$ 
is generated by all elements of the form $(x_k-1)(x_l-1)$. 
In view of the above remarks, it is enough to prove that
\begin{equation}
\label{eq:final}
(x_k-1)(x_l-1)\cdot p \big(\overline{(x_i, x_j)}\big)=0,
\quad \text{for all $i,j,k,l\in [n]$,}
\end{equation}
where the bar denotes the class modulo the second derived subgroup.

Since $n\ge 5$, there is an index $m\in [n]\setminus \{ i,j,k,l \}$. 
Set $v:= (x_l, (x_i, x_j))\in \G^3(F_n)$, and
consider the automorphism $\alpha\colon F_n\to F_n$ 
sending $x_h$ to $x_h$ for $h\ne m$, and $x_m$ to
$vx_m$. By definition \eqref{eq:jfilt}, $\alpha$ belongs 
to $J_n^2= \IA_n'$. 
Using \eqref{eq:bilg}, we obtain:
\[
\alpha((x_m, x_k))= (vx_m, x_k) = 
{}^v(x_m, x_k) \cdot (v, x_k)\equiv (x_m, x_k) 
\cdot (v, x_k)\  \bmod F_n''.
\]
Therefore, 
$\overline{(x_k, v)}=\overline{(x_m, x_k)}-\alpha\cdot \overline{(x_m, x_k)}$, 
and thus $p\big(\overline{(x_k, v)}\big)=0$. 
On the other hand, 
\[
\overline{(x_k, v)}= (x_k-1)\cdot \overline{(x_l, (x_i, x_j))}= 
(x_k-1)(x_l-1)\cdot  \overline{(x_i, x_j)}. 
\]
Since $p$ is $R$-linear, this computation completes the proof. 
\end{proof}

\subsection{The first Betti number of $\IA_n'$}
\label{subsec:nr}
We are now in a position to state and prove the main result of 
this section.  In view of Corollary \ref{cor:jprime}\eqref{jo1}, 
this result settles in the negative a conjecture formulated in \cite{CHP}.

\begin{theorem}
\label{thm:fhp}
If $n\ge 5$, then $\dim_{\C} H_1 (\IA_n', \C)<\infty$. 
\end{theorem}

\begin{proof}
By Corollary \ref{cor:jprime}, we have an exact sequence 
\begin{equation}
\label{eq:iaseq again}
\xymatrix{1\ar[r] & F_n' \ar^(.43){\Ad}[r] 
& \IA_n' \ar[r] & \OA_n' \ar[r] & 1}.
\end{equation}
Furthermore, in view of the discussion from \S\ref{subsec:tg bg}, 
the conjugation action of $\OA_n'$ on $B_{\C}(F_n)=H_1(F_n', \C)$ 
coming from this group extension is given by $\bar{\alpha}\cdot z= 
\alpha_*(z)$, for all $\alpha\in \IA_n' $ and $z\in H_1(F_n', \C)$. 

As is well-known (see e.g.~\cite[p.~202]{HS}), 
extension \eqref{eq:iaseq again} gives rise to an exact sequence
\begin{equation}
\label{eq:h1seq}
\xymatrix{H_1(F_n', \C)_{\IA_n'} \ar[r] & H_1(\IA_n', \C) \ar[r] 
& H_1(\OA_n', \C)\ar[r] & 0}, 
\end{equation}
where, as before, the co-invariants $H_1(F_n', \C)_{\IA_n'}$ 
are taken with respect to the natural action of $\IA_n'$ on 
$H_1(F_n', \C)$.  Of course, the $\C$-vector space 
$H_1(F_n', \C)$ is infinite dimensional; nevertheless, 
the co-invariants under the $\IA_n'$-action 
form a finite-dimensional quotient space, i.e., 
\begin{equation}
\label{eq:b1bfn}
\dim_{\C} H_1(F_n', \C)_{\IA_n'}<\infty.
\end{equation} 

Indeed, it follows from \eqref{eq:linp} that 
$H_1(F_n', \C)_{\IA_n'}$ is a finitely generated 
module over $\C{\Z^n}$.  By Lemma \ref{lem:isq}, 
this module is nilpotent; hence, 
it must be finite-dimensional over $\C$, by
standard commutative algebra.

On the other hand, Theorem \ref{thm:a}\eqref{tai3} gives that 
$\dim_{\C} H_1(\OA_n', \C)<\infty$.  Putting this fact together 
with \eqref{eq:h1seq} and \eqref{eq:b1bfn} finishes the proof.
\end{proof}

This theorem yields further information on the group of 
$\IA$-automorphisms of $F_n$. 

\begin{theorem}
\label{thm:b}
For $n\ge 5$, the following hold.
\begin{enumerate}
\item \label{tbi1}
The characteristic variety $\VV (\IA_n)$ is finite.
\item \label{tbi2}
The Alexander polynomial of $\IA_n$ is a non-zero constant, modulo units.
\item \label{tbi3}
For every subgroup $N$ of $\IA_n$ containing $\IA'_n$, $b_1(N)<\infty$.
\end{enumerate}
\end{theorem}

\begin{proof}
The first assertion follows from Corollary \ref{cor:btest}\eqref{bt2} and 
Theorem \ref{thm:fhp}.  
The other two assertions follow by the same argument as in the 
proof of Theorem \ref{thm:a}, parts \eqref{tai2} and \eqref{tai3}.
\end{proof}

We conclude with some questions raised by 
Theorems \ref{thm:a} and \ref{thm:b}.

\begin{question}
\label{q:oan}
Let $G=\OA_n$ ($n\ge 4$), or $G=\IA_n$ ($n\ge 5$).
\begin{enumerate}
\item \label{qo1} Is $G$ $1$-formal?
\item \label{qo2} Is $\VV(G)=\{1\}$? 
\item \label{qo3} What is $b_1(G')$?
\end{enumerate}
\end{question}

Suppose the answer to questions \eqref{qo1} and \eqref{qo2} 
is yes. Then, by Theorem \ref{thm:cv1}\eqref{v3}, one would be able 
to express the Betti number from question \eqref{qo3} as 
$b_1(G')=\dim_{\C} \B(G)$, which in principle is much easier 
to compute.

\begin{ack}
Much of this work was done at the Centro di Ricerca Matematica 
Ennio De Giorgi in Pisa, in May-June 2010. The authors wish to 
thank the organizers of the Intensive Research Period on 
Configuration Spaces: Geometry, Combinatorics and Topology 
for their warm hospitality, and for providing an excellent 
mathematical environment. We also wish to thank the referee 
for an inspiring and thorough report. 
\end{ack}

\newcommand{\arxiv}[1]
{\texttt{\href{http://arxiv.org/abs/#1}{arxiv:#1}}}
\renewcommand{\MR}[1]
{\href{http://www.ams.org/mathscinet-getitem?mr=#1}{MR#1}}
\newcommand{\doi}[1]
{\texttt{\href{http://dx.doi.org/#1}{doi:#1}}}

\end{document}